\documentclass[10pt]{amsart}

\usepackage{graphicx}
\usepackage[all]{xy}
\usepackage{amssymb}
\usepackage{latexsym}
\usepackage{amsmath}
\usepackage{mathrsfs}
\usepackage{array}
\usepackage{verbatim}
\usepackage{color}

\usepackage{tikz}
\usetikzlibrary{arrows}

\usepackage{hyperref}

\theoremstyle{plain}
\newtheorem{theorem}{Theorem}

\newtheorem{lemma}[theorem]{Lemma}
\newtheorem{corollary}[theorem]{Corollary}

\newtheorem{Def}[theorem]{Definition}

\newcommand{\Db}{{\rm D}^b}
\newcommand{\Cb}{{\rm C}^b}
\newcommand{\Kbp}{{\rm K}^{b,p}}

\renewcommand{\Im}{\mathrm{Im\,}}

\newcommand{\kk}{\mathbf{k}}

\DeclareMathOperator{\Hom}{\mathrm{Hom}}
\DeclareMathOperator{\RHom}{\mathrm{RHom}}
\DeclareMathOperator{\End}{\mathrm{End}}
\DeclareMathOperator{\Aut}{\mathrm{Aut}}
\DeclareMathOperator{\Out}{\mathrm{Out}}
\DeclareMathOperator{\Inn}{\mathrm{Inn}}
\DeclareMathOperator{\Pic}{\mathrm{Pic}}
\DeclareMathOperator{\TrPic}{\mathrm{TrPic}}

\DeclareMathOperator{\Ker}{\mathrm{Ker}}

\DeclareMathOperator{\Id}{\mathrm{Id}}
\DeclareMathOperator{\Sph}{\mathrm{Sph}}

\newcommand{\sB}{\mathcal{B}}

\newcommand{\sT}{\mathcal{T}}
\newcommand{\sX}{\mathcal{X}}
\newcommand{\sS}{\mathcal{S}}
\newcommand{\sI}{\mathcal{I}}

\newcommand{\sO}{\tau}

\newcommand{\lowerz}[1]{\left\lfloor#1\right\rfloor}

\numberwithin{theorem}{section}
\numberwithin{equation}{section}

\title[Groups generated by two twists]{Groups generated by two twists along spherical  sequences}

\author{Y. Volkov}

\address{}

\begin{document}

\begin{abstract}
We describe all groups that can be generated by two twists along spherical sequences in an enhanced triangulated category. It will be shown that with one exception such a group is isomorphic to an abelian group generated by not more than two elements, the free group on two generators or the braid group of one of the types $A_2$, $B_2$ and $G_2$ factorized by a central subgroup. The last mentioned subgroup can be nontrivial only if some specific linear relation between length and sphericity holds. The mentioned exception can occur when one has two spherical sequences of length $3$ and sphericity $2$.
In this case the group generated by the corresponding two spherical twists can be isomorphic to the nontrivial central extension of the symmetric group on three elements by the infinite cyclic group.
Also we will apply this result to give a presentation of the derived Picard group of selfinjective algebras of the type $D_4$ with torsion $3$ by generators and relations.
\end{abstract}

\maketitle

\section{Introduction}

Triangulated categories is a powerful tool that was studied by many mathematicians. They have applications in algebraic geometry, representation theory and many other parts of mathematics. The equivalences and autoequivalences of triangulated categories play here a special role. Motivated by occurrences of braid groups in symplectic geometry and by the Kontsevich's homological mirror conjecture, the authors of \cite{ST} introduced the notion of a twist endofunctor along an object of a good enough triangulated category. They have shown that twists along so called spherical objects are autoequivalences. Later in \cite{Anno} the general notions of a spherical functor and a twist along it were introduced. Note that by results of \cite{Seg} any autoequivalence can be realized as a twist along a spherical functor. The author of \cite{Efimov} considered spherical functors induced by spherical sequences of objects and show that such sequences behave very similar to spherical objects. In particular, one can define a $\Gamma$ configuration of spherical sequences  that under some conditions gives an action of the braid group of the type $\Gamma$. It was shown in the same paper that such an action determined by an $A_n$-configuration of $m$-spherical sequences of length $k$ is faithful for $m\ge 2k$. This result generalizes the original result on the faithfulness of the action of a braid group determined by an $A_n$-configuration of spherical objects proved in \cite{ST}. It was shown also in \cite{BT} that an action determined by a $\Gamma$-configuration of $2$-spherical objects is faithful for $\Gamma=A_n,D_n,E_6,E_7,E_8$. For the derived category of a Ginzburg algebra this result was generalized to an arbitrary $n\ge 2$ in \cite{WC}. At the same time an example from \cite{ST} shows that even $A_3$-configuration of $1$-spherical objects does not have to determine a faithful action of the braid group on $4$ strands. On  the other hand, $m$-spherical objects with $m\le 0$ deserve attention too. They were considered, for example, in \cite{CS1,CS2,CS3,HJY}. 

In this paper we consider the actions generated by two twists along spherical sequences. If there are no morphisms from one sequence to another, then it is known that twist functors commute, and hence the group under consideration is an abelian group generated by two elements. In other cases, we will show that this group is isomorphic to the free group on two generators with several exceptions. These exceptions take place in the case where one spherical sequence has length $k$ and sphericity $m$, the second one has length $rk$ and sphericity $rm$ for some $r\in\{1,2,3\}$, and there are exactly $rk$ morphisms from the first spherical sequence to the second one.
We will show that  the group generated by two autoequivalences under consideration is a factor of the braid group corresponding to the diagram $A_2$ if $r=1$, to the diagram $B_2$ if $r=2$, and to the diagram $G_2$ if $r=3$ by some subgroup $H$.
Except the cases $k=k'=3$, $m=m'=2$ and $k'=2k=4$, $m'=2m=2$, the subgroup $H$ is central and can be nontrivial only if $r=1$ and $3m=4k$, $r=2$ and $2m=3k$ or $r=3$ and $3m=5k$.
In the case $k=k'=3$, $m=m'=2$, the subgroup $H$ is either trivial or is generated as a normal subgroup by the relation equalizing the squares of standard generators. We will show that the last mentioned case really occurs in the derived category of a hereditary algebra of type $D_4$.
Thus, we will give a new example where a braid group action induced by twists along spherical sequences is extremely not faithful.
In the case $k'=2k=4$, $m'=2m=2$, the subgroup $H$ is either trivial or contains the commutator of the braid group. We will show that in the last mentioned case one gets an action of the group $(\mathbb{Z}\times\mathbb{Z})/(2t,-2t)$ for some integer $t$.
We will show also that this occurs in the derived category of a hereditary algebra of type $A_3$. Thus, we will give an example where two twists along spherical sequences generate an abelian group in a nontrivial way.
In particular, for two spherical objects that generate nonabelian group, the group under consideration is isomorphic either to the free group on two generators or to the braid group on $3$ strands even in the case of $1$-spherical objects.

Due to \cite{Ric}, if two derived categories of algebras over a field are equivalent, then there is an equivalence induced by a tensor product with a tilting complex of bimodules.
Such equivalences are closed under composition, and hence give a subgroup of the group of derived autoequivalences. This group is called {\it the derived Picard group} of an algebra and was first introduced in \cite{RZ} and \cite{Ye1}.
Later it was shown in \cite{Ye2} that this group is locally algebraic. Examples of computations of derived Picard groups can be found, for example, in \cite{MY,Miz}.

It was shown in \cite{Rie} that AR quiver of a selfinjective algebra of finite representation is isomorphic to $\mathbb{Z}\Gamma/G$ for $\Gamma\in\{A_n,D_n,E_6,E_7,E_8\}$ and some admissible subgroup $G$ of the automorphism group of $\mathbb{Z}\Gamma$.
Such a group $G$ is cyclic and generated by an element of the form $\tau^{qm_{\Gamma}}\phi$, where $\tau$ is the AR translation, $m_{\Gamma}$ is the Loewy length of the mesh category of $\mathbb{Z}\Gamma$, $q$ is some rational number and $\phi$ is an automorphism of $\Gamma$.
In this case one says that the corresponding selfinjective algebra has the type $(\Gamma, q, r)$, where $r$ is the order of $\phi$. Derived equivalent algebras have the same type. The diagram $\Gamma$ is called the {\it tree type}, the number $q$ is called {\it frequency} and the number $r$ is called {\it torsion order} of the corresponding algebra. It was shown in \cite{Asashiba} that all the possible types are $\big(A_n,\frac{k}{n},1\big)$, $(A_{2t+1},k,2)$, $\Big(D_n,\frac{k}{gcd(n,3)},1\Big)$, $(D_n,k,2)$, $(D_4,k,3)$, $(E_6,k,2)$ and $(E_l,k,2)$, where all numbers are integer and $6\le l\le 8$. Moreover, it was shown in the same work that all the types except the type $(D_{3n},\frac{1}{3},1)$ in characteristic $2$ determine the corresponding algebra uniquely modulo derived equivalence. The type $(D_{3n},\frac{1}{3},1)$ in characteristic $2$ corresponds to two derived equivalence classes of algebras.
The algebras of the type $(A_n,q,1)$ are exactly algebras derived equivalent to selfinjective Nakayama algebras. The derived Picard group for these algebras was described in \cite{VZ2}.
The faithfulness of an action of a braid group determined by an $A_n$-configuration of $0$-spherical objects played a crucial role in this description.
Using our results on the faithfulness of actions of braid groups generated by two spherical twists, we will describe here the derived Picard group of representation finite selfinjective algebras of the type $(D_4,k,3)$. It is is a unique series of algebras with torsion of order $3$. It was considered \cite{GV,GIV}, where its Hochschild cohomology was computed.

\section{Generalized braid groups and spherical twists}

Let us recall the definition of a generalized braid group. Sometimes these groups are called Artin groups.

\begin{Def}{\rm Let $M=(m_{i,j})_{1\le i,j\le s}$ be a symmetric square matrix such that each element $m_{i,j}$ is either some integer not less than $2$ or $\infty$. We associate to this matrix the graph $\Gamma_M$ with $s$ vertices numbered by integers from $1$ to $s$. Two distinct vertices $i$ and $j$ are connected by an edge if $m_{i,j}>2$. If $m_{i,j}>3$, then we also put the number $m_{i,j}$ on the corresponding edge. The {\it braid group of type $\Gamma_M$} is the group with $s$ standard generators $\sigma_1,\dots,\sigma_s$ and relations $(\sigma_i\sigma_j)^{t_{i,j}}\sigma_i^{m_{i,j}-2t_{i,j}}=(\sigma_j\sigma_i)^{t_{i,j}}\sigma_j^{m_{i,j}-2t_{i,j}}$ for all $1\le i<j\le s$ such that $m_{i,j}<\infty$, where $t_{i,j}=\lowerz{\frac{m_{i,j}}{2}}$. Whenever a braid group appears, we denote by $\sigma_i$ its standard generators.
}
\end{Def}

Thus, the braid group of type $1\frac{\phantom{1}\infty\phantom{1}}{} 2$ is the free group $F_2$ on two generators. We will be mainly interested in the braid groups of types $A_2=(1\frac{\phantom{040}}{}2)$, $B_2=(1\frac{\phantom{0}4\phantom{0}}{}2)$ and $G_2=(1\frac{\phantom{0}6\phantom{0}}{}2)$.
Note that all these braid groups have a cyclic center as any braid group of finite type. The center of the braid groups of the types $A_2$ and $G_2$ are generated by $(\sigma_1\sigma_2)^3$ and the center of the braid groups of the type $B_2$ is generated by $(\sigma_1\sigma_2)^2$. We denote the corresponding element generating the center by $\Delta_{A}$, $\Delta_B$ or $\Delta_G$ and the corresponding braid group by $\sB_A$, $\sB_B$ or $\sB_G$ respectively.

Another group that we will meet in this paper is the nontrivial central extension of the symmetric group on three elements by the infinite cyclic group. We will denote this group by $S_3^\mathbb{Z}$. It can be defined by generators and relation by the equality $S_3^\mathbb{Z}=\langle \sigma_1,\sigma_2\mid \sigma_1\sigma_2\sigma_1=\sigma_2\sigma_1\sigma_2,\sigma_1^2=\sigma_2^2\rangle$.

In this paper we will work with {\it algebraic triangulated categories} (see \cite{Efimov} and references there) over some fixed algebraically closed field $\kk$. In fact, all that we need is to take the cones of morphisms between exact functors and we may assume that our triangulated categories are equipped with any enhancement that allows to do this (see \cite{BT}). If $\sT$ is an algebraic triangulated category, then for each $X\in\sT$ we have the derived $\Hom$-functor $\RHom_\sT(X,-):\sT\rightarrow {\rm D}\kk$ with the left adjoint $-\otimes X:{\rm D}\kk\rightarrow \sT$. Here ${\rm D}\kk$ denotes the derived category of the category of $\kk$-linear spaces. If for each $Y\in\sT$ the complex $\RHom_\sT(Y,X)$ has finite dimensional total homology, then the functor $-\otimes X:{\rm D}\kk\rightarrow \sT$ has also the left adjoint $D\RHom_\sT(-,X):\sT\rightarrow\Db\kk$, where $D=\RHom_{\Db\kk}(-,\kk)$ is the usual duality on the bounded derived category $\Db\kk$ of  finite dimensional vector spaces. We will call the object $X\in\sT$ {\it twisting} if for each $Y\in\sT$ the complexes $\RHom_\sT(Y,X)$ and $\RHom_\sT(X,Y)$ have finite dimensional total homology.

Let us recall that a {\it Serre functor} on a triangulated $\sT$ with finite dimensional $\Hom$-spaces is an equivalence $\sS:\sT\rightarrow\sT$ commuting with the shift functor such that there is bifunctorial isomorphism $\phi_{F,G}:\Hom_\sT(F,G)\rightarrow D\Hom_\sT(G,\sS F)$ for any $F,G\in\sT$, where $D=\Hom_\kk(-,\kk)$ is the usual duality.
If a Serre functor exists, it is unique up to isomorphism, but we will not be interested in categories with Serre functor in this paper. Instead of this, we will be interested in objects in a more or less arbitrary triangulated category for which the Serre functor exists in a weak sense.

\begin{Def}{\rm The object $F\in\sT$ {\it admits Serre functor} if $\Hom_\sT(F,G)$ and $\Hom_\sT(G,F)$ are finite dimensional for any $G\in \sT$ and there exist $F'\in\sT$ and $\eta:\Hom_\sT(F,G)\rightarrow\kk$ such that the composition map $$\Hom_\sT(G,F')\times \Hom_\sT(F,G)\rightarrow \Hom_\sT(F,F')\xrightarrow{\eta} \kk$$ is a perfect pairing for any $G\in\sT$. In this case the pair $(F',\eta)$ satisfying this condition is unique modulo isomorphism. We will denote it by $(\sS F,\eta_F)$.
In particular, whenever we write $\sS F$ we mean that $F$ admits Serre functor.
We say that $F$ {\it admits inverse Serre functor} if there exists $F''$ admitting Serre functor such that $\sS F''=F$. Of course, such an object is unique modulo isomorphism and we denote it by $\sS^{-1}F$.
Whenever we write $\sS^k F$ for some $k\in\mathbb{Z}$, we assume that all the required powers of $\sS$ are defined on $F$.
}
\end{Def}

Note that if $F$ and $G$ admit Serre functor, then $F\oplus G$ admits Serre functor and one can set $\sS(F\oplus G)=\sS F\oplus\sS G$ and $\eta_{F\oplus G}=\eta_F\pi_{F}+\eta_G\pi_G$, where $\pi_F$ and $\pi_G$ are canonical projections from
\begin{multline*}\Hom_\sT(F\oplus G,\sS F\oplus\sS G)\\
\cong \Hom_\sT(F,\sS F)\oplus\Hom_\sT(F,\sS G)\oplus\Hom_\sT(G,\sS F)\oplus\Hom_\sT(G,\sS G)
\end{multline*}
to $\Hom_\sT(F,\sS F)$ and $\Hom_\sT(G,\sS G)$ correspondingly. For convenience, we introduce also $\sO:=\sS[-1]$, i.e. by $\sO^tX=\sS^{t} X[-t]$ by definition.


\begin{Def}{\rm The twisting object $E$ is called {\it $m$-spherical} if $\sO E\cong E[m-1]$ and the space $\oplus_{l\in\mathbb{Z}}\Hom_{\sT}(E[l],E)$ is two dimensional with the basis $\Id_E,f^E$, where $f^E\in \Hom_\sT(E[-m],E)$ and $f^E\circ f^E=0$ if $m=0$.
}
\end{Def}

\begin{Def}{\rm An {\it $m$-spherical sequence of length $k\ge 2$} is a collection of twisting objects $E_i$ ($i\in\mathbb{Z}/k\mathbb{Z}$) of the category $\sT$ such that $\sO E_{i}\cong E_{i+1}[m_i-1]$ for some integers $m_i$ with $\sum\limits_{i\in\mathbb{Z}/k\mathbb{Z}}m_i=m$ and
$$\Hom_\sT(E_i[l],E_j)=\begin{cases}
1,&\mbox{ if either $j=i$, $l=0$ or $j=i+1$, $l=-m_i$},\\
0,&\mbox{ otherwise}.
\end{cases}
$$
In the case where $k=2$ and $m=0$ we require also $E_1\not\cong E_0$. For each $i\in\mathbb{Z}/k\mathbb{Z}$ we choose some basic element in $\Hom_\sT(E_i[-m_i],E_{i+1})$ and denote it by $f^E_i$. The collection $(E_0)$ formed by one object is $m$-spherical sequence of length one if $E_0$ is an $m$-spherical object. In this case $f^E_0$ is the element  completing $\Id_{E_0}$ to the basis of $\Hom_{\sT}(E_0,E_0)$ from the definition of an $m$-spherical object.
}
\end{Def}

For an $m$-spherical sequence $(E_0,\dots,E_{k-1})$ of length $k$, we set $E=\oplus_{i\in\mathbb{Z}/k\mathbb{Z}}E_i$. Moreover, we will sometimes call $E$ an $m$-spherical sequence meaning that it is a direct some of members of such a sequence.
One has the adjoint functors $\RHom_\sT(E_i,-)$ and $-\otimes E_i$ for each $i\in\mathbb{Z}/k\mathbb{Z}$. The direct sum of counits of these adjunctions gives a morphism of functors $\oplus_{i\in\mathbb{Z}/k\mathbb{Z}}\RHom_\sT(E_i,-)\otimes E_i\xrightarrow{ev}\Id_\sT$.

\begin{Def}{\rm Given  an $m$-spherical sequence $E$, the cone endofunctor of the morphism $\oplus_{i\in\mathbb{Z}/k\mathbb{Z}}\RHom_\sT(E_i,-)\otimes E_i\xrightarrow{ev}\Id_\sT$ is called a {\it spherical twist} along the sequence $E$ and is denoted by $T_E$.
}
\end{Def}

Let us recall some basic properties of spherical twists. First of all, for any spherical sequence $E$, the functor $T_E$ is an autoequivalence whose inverse is the cocone of the direct sum of units $\Id_\sT\xrightarrow{ev'}\oplus_{i\in\mathbb{Z}/k\mathbb{Z}}D\RHom_\sT(-,E_i)\otimes E_i$. If $(E_0,\dots,E_{k-1})$ is an $m$-spherical sequence and $E'_i\cong E_{i+t}[l_i]$ for some integers $t$ and $l_i$ ($i\in\mathbb{Z}/k\mathbb{Z}$), then the sequence $(E_0',\dots,E_{k-1}')$ is $m$-spherical too and $T_{E'}\cong T_E$. We will write $E'\sim E$ in this case and $E'\not\sim E$ in the opposite case. Note also that if $E$ and $E'$ are spherical sequences and $E_i\cong E_j[l]$ for some integers $i,j,l$, then, 
using the condition $E_{i+1}\cong \sS E_i[-m_i]$, one gets $E'\sim E$. Since in the case $E'\sim E$ the group generated by $T_E$ and $T_{E'}$ coincides with the group generated by $T_E$, we will be concentrated on the case $E'\not\sim E$ in this paper.

Though we are interested in the action of $T_E$ and $T_{E'}$ on the whole category $\sT$, in major part of our proof we will consider their action on a smaller set. Namely, let us denote by $\Sph_{\sT}$ the equivalence classes of spherical sequences modulo the relation $\sim$. It is clear that the action of $T_E$ and $T_{E'}$ on $\sT$ induces an action of $T_E$ and $T_{E'}$ on $\Sph_{\sT}$. For a spherical sequence $F$ we will denote its class in $\Sph_{\sT}$ by $F$ too.

If $m_i$ ($i\in\mathbb{Z}/k\mathbb{Z}$) are numbers from the definition of an $m$-spherical sequence, then
\begin{equation}\label{twistonobj}
T_E(E_{i+1})=E_{i}[1-m_i]\cong \sO^{-1} E_{i+1}
\end{equation}
for $i\in\mathbb{Z}/k\mathbb{Z}$. In particular, $T_E^k(E)=E[k-m]$, and hence the group generated by $T_E$ is an infinite cyclic group if $m\not=k$. In the case $m=k$, the autoequivalence $T_E$ generates either the infinite cyclic group or a cyclic group of some order divisible by $k$.
Another known property of spherical twists is presented in the next lemma.

\begin{lemma}\label{comm} Let $\Phi$ be an autoequivalence of the category $\sT$ and $E$ be an $m$-spherical sequence in $\sT$. Then $\Phi T_E$ is naturally isomorphic to $T_{\Phi E}\Phi$.
\end{lemma}

If $E$ is an $m$-spherical sequence of length $k$ and $E'$ is an $m'$-spherical sequence of length $k'$ such that $\Hom_\sT(E,E'[l])=0$ for any $l\in\mathbb{Z}$, then $T_E$ and $T_{E'}$ commute, and hence generate an abelian group. This group is isomorphic to $\mathbb{Z}\times\mathbb{Z}$ if $m\not=k$ and $m'\not=k'$ since $T_E^{ak}T_{E'}^{bk'}(E)=E[a(k-m)]$ and $T_E^{ak}T_{E'}^{bk'}(E')=E'[b(k'-m')]$. If $m=k$ or $m'=k'$, then the group generated by $T_E$ and $T_{E'}$ can be not isomorphic to $\mathbb{Z}\times\mathbb{Z}$ and can be even finite. We are not going to consider this problems in details here and will concentrate on the case where $\oplus_{l\in\mathbb{Z}}\Hom_\sT(E,E'[l])\not=0$.

Another situation considered in \cite{Efimov} occurs when $k=k'$, $m=m'$ and $\sum\limits_{l\in\mathbb{Z}}\dim_\kk\Hom_\sT(E,E'[l])=k$. In this case it can be deduced from Lemma \ref{comm} that $T_E$ and $T_{E'}$ satisfy the braid relation $T_ET_{E'}T_E=T_{E'}T_ET_{E'}$. Let us also recall the main result of \cite{Efimov} that generalizes the main result of \cite{ST}. We say that the $m$-spherical sequences $E^1,\dots,E^n$ of length $k$  form an {\it $A_n$-configuration} if $$\sum\limits_{l\in\mathbb{Z}}\dim_\kk\Hom_\sT(E^i,E^j[l])=\begin{cases}k,&\mbox{ if $|i-j|=1$},\\0,&\mbox{ if $|i-j|>1$}.\end{cases}$$
Then the main result of \cite{Efimov} says that if the $m$-spherical sequences $E^1,\dots,E^n$ of length $k$  form an $A_n$-configuration and $m\ge 2k$, then the group generated by $T_{E^1},\dots,T_{E^n}$ is isomorphic to the braid group on $n+1$ strands, i.e. the braid group of type $A_n$, where
$
A_n=\big(1\frac{\phantom{040}}{}2\frac{\phantom{040}}{}\cdots\frac{\phantom{040}}{}(n-1)\frac{\phantom{040}}{}n\big).
$

The main result of this paper is the following theorem.

\begin{theorem}\label{main}
Let $E$ be an $m$-spherical sequence of length $k$ and $E'$ be an $m'$-spherical sequence of length $k'$ such that $k\le k'$, $\sum\limits_{l\in\mathbb{Z}}\dim_\kk\Hom_\sT(E,E'[l])\not=0$ and $E\not\sim E'$.
\begin{enumerate}
\item\label{a2} Suppose that $k'=k$, $m'=m$ and $\sum\limits_{l\in\mathbb{Z}}\dim_\kk\Hom_\sT(E,E'[l])=k$.
\begin{itemize}
\item If $3m\not=4k$ and $(m,k)\not=(2,3)$, then the group generated by $T_E$ and $T_{E'}$ is isomorphic to the braid group of type $A_2$.
\item If $3m=4k$, then  the group generated by $T_E$ and $T_{E'}$ is isomorphic to the factor group of the braid group of type $A_2$ by the cyclic group generated by the element $\Delta_A^{t\frac{k}{gcd(k,3)}}$ for some $t\in\mathbb{Z}$.
\item If $m=2$ and $k=3$, then the group generated by $T_E$ and $T_{E'}$ is isomorphic either to the braid group of type $A_2$ or to the group $S_3^\mathbb{Z}$.
\end{itemize}
\item\label{b2} Suppose that $k'=2k$, $m'=2m$ and $\sum\limits_{l\in\mathbb{Z}}\dim_\kk\Hom_\sT(E,E'[l])=2k$.
\begin{itemize}
\item If $2m\not=3k$ and $(m,k)\not=(1,2)$, then the group generated by $T_E$ and $T_{E'}$ is isomorphic to the braid group of type $B_2$.
\item If $2m=3k$, then  the group generated by $T_E$ and $T_{E'}$ is isomorphic to the factor group of the braid group of type $B_2$ by the cyclic group generated by the element $\Delta_B^{t\frac{2k}{gcd(k-2,4)}}$ for some $t\in\mathbb{Z}$.
\item If $m=1$ and $k=2$, then the group generated by $T_E$ and $T_{E'}$ is isomorphic either to the braid group of type $B_2$ or to the group $(\mathbb{Z}\times \mathbb{Z})/(2t,-2t)$ for some integer $t$.
\end{itemize}
\item\label{g2} Suppose that $k'=3k$, $m'=3m$ and $\sum\limits_{l\in\mathbb{Z}}\dim_\kk\Hom_\sT(E,E'[l])=3k$.
\begin{itemize}
\item If $3m\not=5k$, then the group generated by $T_E$ and $T_{E'}$ is isomorphic to the braid group of type $G_2$.
\item If $3m=5k$, then  the group generated by $T_E$ and $T_{E'}$ is isomorphic to the factor group of the braid group of type $G_2$ by the cyclic group generated by the element $\Delta_G^{tk}$ for some $t\in\mathbb{Z}$.
\end{itemize}
\item\label{last} In all the remaining cases $T_E$ and $T_{E'}$ generate the free group $F_2$ on two generators.
\end{enumerate}
\end{theorem}

Assume that $E$ and $E'$ are twisting objects such that $\sS^k(E)\cong E[m]$ and $\sS^{k'}(E')\cong E'[m']$. If $\Hom_\sT(E,E'[l])\not=0$, then
\begin{multline*}
\Hom_\sT(E,E'[l+km'-k'm])=\Hom_\sT(E[k'm],E'[l+km'])\\
=\dim_\kk(\sS^{kk'} E,\sS^{kk'} E'[l])=\dim_\kk\Hom_\sT(E,E'[l])\not=0.
\end{multline*}
Thus, if $k'm\not=km'$, then $\Hom_\sT(E,E'[l])=0$ for any $l\in\mathbb{Z}$.

\section{Free group actions}

In this section we will prove the last assertion of Theorem \ref{main}. During this section we assume that $E$ and $E'$ are two spherical sequences such that $E\not\sim E'$.
Let us start with an easy observation.

\begin{lemma}\label{dims}
The values of $a_i=\sum\limits_{l\in\mathbb{Z}}\dim_\kk\Hom_\sT(E_i,E'[l])$ and $b_{i'}=\sum\limits_{l\in\mathbb{Z}}\dim_\kk\Hom_\sT(E,E'_{i'}[l])$ do not depend on $i\in\mathbb{Z}/k\mathbb{Z}$ and $i'\in\mathbb{Z}/k'\mathbb{Z}$.
Moreover, $ka_i=k'b_{i'}$.
\end{lemma}
\begin{proof} Note that.
\begin{multline*}\dim_\kk\Hom_\sT(E_i,E_{i'}'[l])=\dim_\kk\Hom_\sT(\sS E_i,\sS E_{i'}'[l])\\
=\dim_\kk\Hom_\sT(E_{i+1},E_{i'+1}'[l+t_{i,i'}])
\end{multline*}
for some integers $t_{i,i'}$ ($i\in\mathbb{Z}/k\mathbb{Z}$,d $i'\in\mathbb{Z}/k'\mathbb{Z}$).
Taking the sum over all $l$ and $i'$ we get $a_{i+1}=a_i$ and taking the sum over all $l$ and $i$ we get $b_{i'+1}=b_{i'}$. Now, taking the sum over all $l$, $i$ and $i'$, we get $ka_i=k'b_{i'}$.
\end{proof}

We will denote the numbers $a_i$ from the lemma by $a_{E,E'}$. Then Lemma \ref{dims} guarantees that $a_{E,E'}$ is well defined. Note also that it follows from the equality $\dim_\kk\Hom_\sT(E'_{i'},E_{i+1}[m_i-l])=\dim_\kk\Hom_\sT(E_i,E'_{i'}[l])$ that the number $b_{i'}$ from the lemma is equal to $a_{E',E}$.

Let us now give some facts and constructions related to spherical sequences and twists along them.
Suppose that $E\in\sT$ is an object admitting Serre functor such that $\End_\sT(E)$ is local with the maximal ideal $\sI$. Suppose also that $f:E\rightarrow \sS E$ is such that $f\sI=0$. If $g:X\rightarrow E$ is such that $fg\not=0$, then $g$ is a split epimorphism, and hence $E$ is a direct summand of $X$.
Indeed, if $fg\not=0$, then there is some $g':E\rightarrow X$ such that $fgg'\not=0$, and hence $gg'\in\End_\sT(E)$ is not annihilated by $f$. Then $gg'\not\in\sI$, and hence this element is invertible.
Analogously, if $h:\sS\rightarrow X$ is such that $hf\not=0$, then $\sS E$ is a direct summand of $X$.
In particular, if $E$ and $E'$ are a spherical sequences such that $E'\not\sim E$, then $ff_i^{E}=f_i^{E}[m_i]g=0$ for any integer $l$ and any $f:E\rightarrow E'[l]$ and $g:E'[l]\rightarrow E$.
For a spherical sequence $E$, let us introduce $f^{E}=\oplus_{i\in\mathbb{Z}/k\mathbb{Z}}f^{E}_i:\sS^{-1} E\rightarrow E$. Note that $\sS^tE\sim E$ is a spherical sequence for any integer $t$, and hence $f^{\sS^t E}$ makes sense.

Let us now describe how one can construct $T_E^s(E')$ in the case where $E'\not\sim E$.
First we have a triangle
$\bigoplus\limits_{i\in\mathbb{Z}/k\mathbb{Z}}\bigoplus\limits_{j=1}^{a_{E,E'}}E_i[r_{i,j}]\rightarrow E'\rightarrow T_{E}E'\rightarrow \bigoplus\limits_{i\in\mathbb{Z}/k\mathbb{Z}}\bigoplus\limits_{j=1}^{a_{E,E'}}E_i[r_{i,j}+1]$
for some $r_{i,j}\in\mathbb{Z}$ ($i\in\mathbb{Z}/k\mathbb{Z}$, $1\le j\le a_{E,E'}$). Note that this triangle can be rewritten as 
\begin{equation}\label{tw1}
\bigoplus\limits_{t=1}^{a_{E,E'}}{E^t}\xrightarrow{\epsilon_{E,E'}^0} E'\xrightarrow{\rho_{E,E'}^0} T_{E}E'\xrightarrow{\theta_{E,E'}^0}\bigoplus\limits_{t=1}^{a_{E,E'}}{E^t}[1],
\end{equation}
where ${E^t}\sim E$ for each $1\le t\le a_{E,E'}$. This triangle satisfies the following properties. For any $i\in\mathbb{Z}/k\mathbb{Z}$, $l\in\mathbb{Z}$, $f:E_i[l]\rightarrow E'$ and $g:T_EE'\rightarrow E_i[l]$ there are $\bar f$ and $\bar g$ such that $f=\epsilon_{E,E'}^0\bar f$ and $g=\bar g \theta_{E,E'}^0$.
Applying the functor $T_{E}^s$ to this triangle, we get the triangle
\begin{equation}\label{twgen}
\bigoplus\limits_{t=1}^{a_{E,E'}}\sO^{-s}{E^{t}}\xrightarrow{\epsilon_{E,E'}^{s}} T_{E}^{s}E'\xrightarrow{\rho_{E,E'}^s} T_{E}^{s+1}E'\xrightarrow{\theta_{E,E'}^{s}}\bigoplus\limits_{t=1}^{a_{E,E'}}\sO^{-s}{E^{t}}[1]
\end{equation}
satisfying the analogous property. Moreover, we may assume that $(\theta_{E,E'}^{s-1}\epsilon_{E,E'}^{s})[-1]=\bigoplus\limits_{t=1}^{a_{E,E'}}f^{\sO^{1-s}E^t}$ for each $s\in\mathbb{Z}$. This equality can be justified for $s=0$ and then proceeded via the functor $T_{E}^s$.

Later we will need the following standard lemma.

\begin{lemma}\label{imimdual} For any $X\in\sT$ and any integer $s$ one has
$$\dim_\kk\Im\Hom_\sT(\theta_{E,E'}^{s},X)=\dim_\kk\Im\Hom_\sT(X,\theta_{\sS E,\sS E'}^{s}).$$	
\end{lemma}
\begin{proof} Let us denote $\bigoplus\limits_{t=1}^{a_{E,E'}}\sO^{-s}{E^{t}}[1]$ by $H$. Since the triangle \eqref{twgen} can be obtained from \eqref{tw1} by applying an autoequivalence, it is enough to prove the required equality for $s=0$.
It is not difficult to check the required equality for direct summands of $T_{E}E'$, and so we will assume that $X$ does not have such direct summands.
Then it is enough to show that
$$
\Im\Hom_\sT(\theta_{E,E'}^{0},X)=\{f\in\Hom_\sT(T_{E}E',X)\mid \Ker\Hom_\sT(X,\theta_{\sS E,\sS E'}^{0})f=0\}.
$$
We will denote the set on the right hand side by $\Ker\Hom_\sT(X,\theta_{\sS E,\sS E'}^{s})^{\perp}$.

Let us consider a morphism of the form $g\theta_{E,E'}^{0}:T_EE'\rightarrow X$. Suppose that there is some $h:X\rightarrow T_{\sS E}\sS E'$ such that $hg\theta_{E,E'}^{0}\not=0$ and $\theta_{\sS E,\sS E'}^{0}h=0$.
Then we have some morphism $u:T_{\sS E}\sS E'\rightarrow \sS H$ such that $uhg\not=0$, and hence also a morphism $\bar u:\sS H\rightarrow \sS H$ such that $u=\bar u\theta_{\sS E,\sS E'}^{0}$. We get a contradiction, because $0\not=uhg=\bar u\theta_{\sS E,\sS E'}^{0}hg=0$. Thus, we have $\Im\Hom_\sT(\theta_{E,E'}^{0},X)\subset \Ker\Hom_\sT(X,\theta_{\sS E,\sS E'}^{s})^{\perp}$.

Let us now pick some $f\in \Ker\Hom_\sT(X,\theta_{\sS E,\sS E'}^{s})^{\perp}$.
If $f\not\in\Im\Hom_\sT(\theta_{E,E'}^{0},X)$, then $f\rho_{E,E'}^{0}\not=0$. Then there is a morphism $g:X\rightarrow \sS E'$ such that $gf\rho_{E,E'}^{0}\not=0$. By our assumptions we have $\rho_{\sS E,\sS E'}^{0}gf=0$, and hence there is $h:T_EE'\rightarrow \sS H[-1]$ such that $gf=\epsilon_{\sS E,\sS E'}^{0}h$. Then there is $\bar h:H\rightarrow \sS H[-1]$ such that $h=\bar h \theta_{E,E'}^{0}$. We get a contradiction, because $0\not=gf\rho_{E,E'}^{0}=\epsilon_{\sS E,\sS E'}^{0}h\rho_{E,E'}^{0}=\epsilon_{\sS E,\sS E'}^{0}\bar h \theta_{E,E'}^{0}\rho_{E,E'}^{0}=0$. Thus, we have $\Ker\Hom_\sT(X,\theta_{\sS E,\sS E'}^{s})^{\perp}\subset \Im\Hom_\sT(\theta_{E,E'}^{0},X)$.
\end{proof}

For twisting objects $F,G\in\sT$, let us set $$u_F(G)=\sum\limits_{l\in\mathbb{Z}}\dim_\kk\Hom_\sT(F,G[l]).$$
If $E\sim F$ are two spherical sequences, then $T_E(G)=T_{F}(G)$ and $T_G(E)=T_G(F)$ for any $G$. In particular, the notation $u_{F}(G)$ makes sense if $F$ or $G$ is considered as an element of $\Sph_{\sT}$.
It is clear also that $T_F(G)=T_G(\sS F)$ whenever $F$ admits Serre functor and that for any equivalence $\Phi$ one has $T_{\Phi F}(\Phi G)=T_F(G)$. Since $E\sim \sS E$ for any spherical sequence $E$, we have  that $u_G(E)=u_G(\sS E)=u_E(G)$. Applying \eqref{twistonobj}, we get also
$u_E(T_EG)=u_{T_E^{-1}E}(G)=u_E(G).$ The crucial technical ingredient of our proof is the following lemma.

\begin{lemma}\label{inequal}
Suppose that $a_{E,E'}a_{E',E}\ge 4$ and $X\in \Sph_\sT$ is such that $u_{E'}(X)\le \frac{a_{E,E'}a_{E',E}-2}{a_{E',E}}u_E(X)$. Then $u_{E'}(T_E^sX)\ge a_{E,E'}u_E(X)-u_{E'}(X)$ for any nonzero integer $s$. Moreover, the last mentioned inequality cannot be an equality if $X\sim E'$.
\end{lemma}
\begin{proof} We will prove the required condition simultaneously for all $X$ satisfying the conditions of the lemma. Applying the functor $\Hom_{\sT}(X,-)$ to \eqref{tw1}, we get pieces of a long exact sequence of the form
\begin{multline*}
\Hom_{\sT}(X,T_{E}E'[l-1])\xrightarrow{\Hom_\sT(X,\theta_{E,E'}^0[l-1])} \bigoplus\limits_{t=1}^{a_{E,E'}}\Hom_{\sT}(X,{E^t}[l])\\
\xrightarrow{\Hom_\sT(X,\epsilon_{E,E'}^0[l])} \Hom_{\sT}(X,E'[l]).
\end{multline*}
Since ${E^t}\not\sim E'$ for any $1\le t\le a_{E,E'}$, the map $\Hom_\sT(E',\epsilon_{E,E'}^0)$ is not surjective. Then we get the inequalities
\begin{multline*}
\dim_\kk\Im\Hom_\sT(X,\theta_{E,E'}^0[l-1])\\
\ge \sum\limits_{t=1}^{a_{E,E'}}\dim_\kk\Hom_{\sT}(X,{E^t}[l])-\dim_\kk\Hom_{\sT}(X,E'[l]),
\end{multline*}
one of which is not an equality if $X\sim E'$.
Taking the sum over all $l\in\mathbb{Z}$, we get the inequality
$$
\sum\limits_{l\in\mathbb{Z}}\dim_\kk\Im\Hom_\sT(X,\theta_{E,E'}^0[l])\ge a_{E,E'}u_E(X)-u_{E'}(X).
$$
Moreover, this inequality cannot turn into an equality if $X\sim E'$.
Note also that applying the functor $\Hom_{\sT}(-,X)$ to the triangle \eqref{twgen} with $s=-1$ one gets in analogous way the inequality
$$
\sum\limits_{l\in\mathbb{Z}}\dim_\kk\Im\Hom_\sT(\epsilon_{E,E'}^{-1}[l],X)\ge a_{E,E'}u_E(X)-u_{E'}(X).
$$
that  cannot turn into an equality if $X\sim E'$.

For $s>0$ let us introduce
$$
A_s=\sum\limits_{l\in\mathbb{Z}}\dim_\kk\Im\Hom_\sT(X,\theta_{E,E'}^{s-1}[l])\mbox{ and }B_s=\sum\limits_{l\in\mathbb{Z}}\dim_\kk\Im\Hom_\sT(\epsilon_{E,E'}^{-s}[l],X).
$$
We will prove by induction on $s$ that $A_s,B_s\ge a_{E,E'}u_E(X)-u_{E'}(X)$ for $s>0$ and the inequality is strict if $X\sim E'$. Note that $u_{E'}(T_E^{-s}X)\ge A_s$ and $u_{E'}(T_E^{s}X)\ge B_s$, and hence we will get the assertion of the lemma.
Everything is already done for $s=1$. Suppose that we have already proved the required equality for $s$. In particular, we have
$u_{E'}(T_E^{\pm s}E')>a_{E,E'}u_E(E')-u_{E'}(E')=\frac{a_{E,E'}a_{E',E}-2}{a_{E',E}}u_{E}(T_E^{\pm s}E')$, and hence $X\not\sim T_E^{\pm s}E'$.

Let us consider a nonzero morphism $f:X\rightarrow \bigoplus\limits_{t=1}^{a_{E,E'}}\sO^{1-s}E^t[l+1]$ from $\Im\Hom_{\sT}(X,\theta_{E,E'}^{s-1}[l])$. Let us pick some $\phi(f):\bigoplus\limits_{t=1}^{a_{E,E'}}{\sO^{-s}E^{t}}[l]\rightarrow X$ such that $\eta_{\bigoplus\limits_{t=1}^{a_{E,E'}}{\sO^{-s}E^{t}}[l]}\big(f\phi(f)\big)\not=0$.
Then we get the diagram
\begin{center}
\begin{tikzpicture}[node distance=1.4cm]
 \node(A) {$\bigoplus\limits_{t=1}^{a_{E,E'}}{\sO^{-s}E^{t}}[l]$};
\node(A') [below  of=A] {$T_{E}^sE'[l]$};
\node(A'') [above  of=A] {$T_{E}^{s+1}E'[l-1]$};
 \node(Ar) [right  of=A] {};
 \node(Arr) [right  of=Ar] {};
 \node(Arrr) [right  of=Arr] {};
 \node(B) [right  of=Arrr] {$X$};
\node(B') [below  of=B] {$\bigoplus\limits_{t=1}^{a_{E,E'}}{\sO^{1-s}E^{t}}[l+1]$};

\draw [->,>=stealth'] (A')  to node[above]{\tiny$\theta_{E,E'}^{s-1}[l]$ } (B') ; 

\draw [->,>=stealth'] (A)  to node[right]{\tiny$\epsilon_{E,E'}^{s}[l]$ } (A') ; 
\draw [->,>=stealth'] (A)  to node[above=8, right]{\tiny$\phi(f)$ } (B) ; 
\draw [->,>=stealth'] (B)  to node[right]{\tiny$f$ } (B') ; 
\draw [->,>=stealth'] (B)  to node[above]{\tiny$\bar f$ } (A') ; 
\draw [->,>=stealth'] (A'')  to node[right]{\tiny$\theta_{E,E'}^{s}[l-1]$ } (A) ; 
\end{tikzpicture}
\end{center}
with $\bar f$ such that $\theta_{E,E'}^{s-1}[l]\bar f=f$. Suppose that $\phi(f)\theta_{E,E'}^{s}[l-1]=0$. Then we have $\phi(f)=g\epsilon_{E,E'}^{s}[l]$ for some $g:T_{E}^sE'[l]\rightarrow X$, and hence $f\phi(f)=\theta_{E,E'}^{s-1}[l]\bar fg\epsilon_{E,E'}^{s}[l]$. Since $X\not\sim T_{E}^sE'$ and both of them are spherical sequences, the morphism $\bar fg$ is annihilated by any morphism $h:Y\rightarrow T_{E}^sE'[l]$ if $Y$ and $T_{E}^sE'[l]$ do not have isomorphic nonzero direct summands. In particular, we have $fg\epsilon_{E,E'}^{s}[l]=0$, because $T_{E}^sE'\not\sim E$.

Now, if $f_1,\dots,f_p$ is a basis of $\Im\Hom_{\sT}(X,\theta_{E,E'}^{s-1}[l])$, then we can choose $\phi(f_1),\dots,\phi(f_p)$ in such a way that $\eta_{\bigoplus\limits_{t=1}^{a_{E,E'}}{\sO^{-s}E^{t}}[l]}\big(f_i\phi(f_j)\big)=0$ for $1\le i,j\le p$, $i\not=j$ and get $p$ linearly independent elements $\phi(f_1)\theta_{E,E'}^{s}[l-1],\dots,\phi(f_p)\theta_{E,E'}^{s}[l-1]$ of $\Im\Hom_\sT(\theta_{E,E'}^s[l-1],X)$. Then we have
$$	
\dim_\kk\Im\Hom_\sT(\theta_{E,E'}^s[l-1],X)\ge\dim_\kk\Im\Hom_{\sT}(X,\theta_{E,E'}^{s-1}[l]).
$$
Taking the sum over all $l\in\mathbb{Z}$ and using Lemma \ref{imimdual}, we get
$$A_{s+1}\ge A_s\ge a_{E,E'}u_E(X)-u_{E'}(X).$$
Moreover, by induction hypothesis, the inequality is strict if $X\sim E'$.
 The proof of the inequality for $B_s$ is dual.
\end{proof}

To prove the last assertion of Theorem \ref{main}, we will use the ping-pong lemma.
To apply this lemma we need to  introduce two disjoint nonempty sets $\sX,\sX'\subset \Sph_{\sT}$ such that $T_E^s(\sX')\subset \sX$ and $T_{E'}^s(\sX)\subset \sX'$ for any nonzero integer $s$.
In fact, we will prove that the action on the free group on two generators induced by twists along $E$ and $E'$ is faithful not only on $\Sph_{\sT}$, but even on the union of orbits of $E$ and $E'$ in $\Sph_{\sT}$ under the action of the group generated by $T_E$ and $T_{E'}$.
We denote the last mentioned subset of $\Sph_{\sT}$ by $\Sph_{E,E'}$.
By this reason we define
\begin{multline*}\sX=\left\{X\in\Sph_{E,E'}\mid u_{E'}(X)>\frac{a_{E,E'}}{2}u_E(X)\right\}\mbox{ and }\\
\sX'=\left\{X\in\Sph_{E,E'}\mid u_E(X)>\frac{a_{E',E}}{2}u_{E'}(X)\right\},
\end{multline*}
It follows directly from the definitions of $\sX$ and $\sX'$ that $\sX\cap \sX'=\varnothing$ if $a_{E,E'}a_{E',E}\ge 4$.

\begin{corollary}\label{free}
If $a_{E,E'}a_{E',E}\ge 4$, then $T_E$ and $T_{E'}$ generate a free subgroup on two generators in the group of permutations of $\Sph_{E,E'}$.
\end{corollary}
\begin{proof} By Lemma \ref{inequal}, we have $$u_{E'}(T_EE')>a_{E,E'}u_E(E')-u_{E'}(E')=\frac{a_{E,E'}a_{E',E}-2}{a_{E',E}}u_{E}(E')\ge \frac{a_{E,E'}}{2}u_E(T_EE'),$$ and hence $T_EE'\in\sX$, i.e. $\sX$ is nonempty. The non emptiness of $\sX'$ will follow from the condition $T_{E'}\sX\subset \sX'$ that we will prove below.

If $X\in \sX'$, then $u_{E'}(X)<\frac{2}{a_{E',E}}u_E(X)\le \frac{a_{E,E'}a_{E',E}-2}{a_{E',E}}u_E(X)$, and hence by Lemma \ref{inequal} one has
$$u_{E'}(T_E^sX)\ge a_{E,E'}u_E(X)-u_{E'}(X)>\frac{a_{E,E'}a_{E',E}-2}{a_{E',E}}u_{E}(X)\ge \frac{a_{E,E'}}{2}u_{E}(T_E^sX)$$
for any nonzero integer $s$. Thus, $T_E^s\sX'\subset \sX$ for any nonzero integer $s$. Analogously, one has $T_{E'}^s\sX\subset \sX'$. Thus, the required assertion follows from the ping-pong lemma.
\end{proof}

Now we are ready to prove the last item of Theorem \ref{main}.

\begin{proof}[Proof of item \eqref{last} of Theorem \ref{main}] By Lemma \ref{free} it is enough to prove that $a_{E,E'}a_{E',E}\ge 4$ if none of the items \eqref{a2}, \eqref{b2} and \eqref{g2} holds. Since $ka_{E,E'}=k'a_{E',E}=\sum\limits_{l\in\mathbb{Z}}\dim_\kk\Hom_\sT(E,E'[l])$, one has $a_{E,E'},a_{E',E}\ge 2$ if $k\nmid k'$ or $\sum\limits_{l\in\mathbb{Z}}\dim_\kk\Hom_\sT(E,E'[l])>k'$. If $\sum\limits_{l\in\mathbb{Z}}\dim_\kk\Hom_\sT(E,E'[l])=k'=kr$, then $a_{E',E}=1$ and $a_{E,E'}=r$. Thus, one has $a_{E,E'}a_{E',E}\ge 4$ if $r\not\in\{1,2,3\}$. Finally, the equality $m'=mr$ follows from the condition $k'm=km'$.
\end{proof}

Before finishing this section, we will give one more result following from Lemma \ref{inequal} and the ping-pong lemma. We will need this result in the next section.

\begin{corollary}\label{free3}
Let $E$, $E'$ and $E''$ be $m$-spherical sequences of length $k$ such that $a_{E,E'}=a_{E,E''}=a_{E',E''}=3$ and there exists a triangle
$
E\rightarrow E'\rightarrow E''.
$
Let us denote by $\Sph_{E,E',E''}$ the union of orbits of $E$, $E'$ and $E''$ in $\Sph_{\sT}$ under the action of the group generated by $T_E$, $T_{E'}$ and $T_{E''}$.
Then the subgroup of the group of permutations of $\Sph_{E,E',E''}$ generated by $T_E$, $T_{E'}$ and $T_{E''}$ is isomorphic to the free group on three generators.
\end{corollary}
\begin{proof} Let us define 
\begin{multline*}
\sX=\{X\in\Sph_{E,E',E''}\mid u_E(X)<\min\big(u_{E'}(X), u_{E''}(X)\big)\},\\
\sX'=\{X\in\Sph_{E,E',E''}\mid u_{E'}(X)<\min\big(u_{E}(X), u_{E''}(X)\big)\},\\
\sX''=\{X\in\Sph_{E,E',E''}\mid u_{E''}(X)<\min\big(u_{E}(X), u_{E'}(X)\big)\}.
\end{multline*}
It is clear that $\sX$, $\sX'$ and $\sX''$ are pairwise disjoint.
Since $u_{E'}(E)=u_{E''}(E)=\frac{3}{2}u_E(E)>u_E(E)$, we have $E\in\sX$, and hence $\sX$ is nonempty. Analogously $\sX'$ and $\sX''$ are nonempty too.

It follows from the existence of the triangle $E\rightarrow E'\rightarrow E''$ that $u_E(X)\le u_{E'}(X)+u_{E''}(X)$, $u_{E'}(X)\le u_{E}(X)+u_{E''}(X)$ and $u_{E''}(X)\le u_{E}(X)+u_{E'}(X)$ for any $X\in\sT$.
It follows from Lemma \ref{inequal} that, for $X\in\sX'$ and a nonzero integer $s$, one has
$$u_{E'}(T_E^sX)\ge a_{E,E'}u_E(X)-u_{E'}(X)> 3u_E(X)-u_E(X)>u_E(X)$$ and $$u_{E''}(T_E^sX)\ge a_{E,E''}u_E(X)-u_{E''}(X)\ge 3u_E(X)-\big(u_E(X)+u_{E'}(X)\big)>u_E(X),$$ i.e. $T_E^sX\in\sX$. Analogously one can show that $T_E^s\sX''\in\sX$, $T_{E'}^s(\sX\cup\sX'')\in\sX'$ and $T_{E''}^s(\sX\cup\sX')\in\sX''$ for any nonzero integer $s$. Thus, the required assertion follows from the ping-pong lemma.
\end{proof}

\section{Actions of generalized braid groups}

In this section we will prove the remaining assertions of Theorem \ref{main}. Thus, we assume during this section that we are in the settings of Theorem \ref{main} and, moreover, $\sum\limits_{l\in\mathbb{Z}}\dim_\kk\Hom_\sT(E,E'[l])=k'=kr$ and $m'=mr$ for some $r\in\{1,2,3\}$. We also set $\Gamma=A_2$ if $r=1$, $\Gamma=B_2$ if $r=2$ and $\Gamma=G_2$ if $r=3$.

Let us first prove that $T_E$ and $T_{E'}$ satisfy the corresponding braid relation.
To do this we adjust degrees of morphisms between $E$ and $E'$. Since we can apply arbitrary shifts to direct summands of $E$ and $E'$ and shift the indices in the enumeration of summands of $E'$, we may assume that 
$E$ and $E'$ are justified in such a way that $\Hom_{\sT}(E_i,E'_{i'}[l])\not =0$ if and only if $k\mid i-i'$ and $l=0$. Then $\Hom_{\sT}(E'_{i'},E_{i}[l])\not =0$ if and only if $k\mid i-1-i'$ and $l=m_{i-1}$, where $m_i$ ($i\in\mathbb{Z}/k\mathbb{Z}$) are numbers from the definition of a spherical sequence. Note that for $i'\in\mathbb{Z}/k'\mathbb{Z}$, one has $m_{i'}'=m_{i'}$, where $m_{i'}'$ are the corresponding numbers for $E'$.
For $i\in\mathbb{Z}$, let $h_{i}$ be some nonzero element of $\Hom_\sT(E_i,E_{i}')$ and $g_{i}$ be some nonzero element of $\Hom_\sT(E_{i-1}'[-m_{i-1}],E_{i})$.

\begin{lemma}\label{actbr} Suppose that $\sum\limits_{l\in\mathbb{Z}}\dim_\kk\Hom_\sT(E,E'[l])=k'=kr$ and $m'=mr$ for some $r\in\{1,2,3\}$.
\begin{enumerate}
\item If $r=1$, then $T_ET_{E'}T_E=T_{E'}T_{E}T_{E'}$.
\item If $r=2$, then $(T_ET_{E'})^2=(T_{E'}T_{E})^2$.
\item If $r=3$, then $(T_ET_{E'})^3=(T_{E'}T_{E})^3$.
\end{enumerate}
\end{lemma}
\begin{proof} The first case is known, see \cite{Efimov}. Due to Lemma \ref{comm}, it is enough to prove that $(T_ET_{E'}T_E)E'\sim E'$ if $r=2$ and $(T_ET_{E'}T_ET_{E'}T_E)E'\sim E'$ if $r=3$.

If $r=2$, then, for $0\le i\le 2k-1$, we have
\begin{multline*}
T_EE_{i}'\cong cone\left(E_i\xrightarrow{h_{i}}E_i'\right),\\
T_{E'}\,cone\left(E_i\xrightarrow{h_{i}}E_i'\right)\cong cone\left(E_{i+k-1}'[-m_{i-1}]\xrightarrow{g_{i+k}}E_i\right)[1],\\
T_{E}\,cone\left(E_{i+k-1}'[-m_{i-1}]\xrightarrow{g_{i+k}}E_i\right)[1]\cong E_{i+k-1}'[2-m_{i-1}].
\end{multline*}
If $r=3$, then, for $0\le i\le 3k-1$, we have
\begin{multline*}
T_EE_{i}'\cong cone\left(E_i\xrightarrow{h_{i}}E_i'\right),\\
T_{E'}\,cone\left(E_i\xrightarrow{h_{i}}E_i'\right)\cong cone\left((E_{i+k-1}'\oplus E_{i+2k-1}')[-m_{i-1}]\xrightarrow{\scriptsize\begin{pmatrix}g_{i+k}&g_{i+2k}\end{pmatrix}}E_i\right)[1],\\
T_{E}\,cone\left((E_{i+k-1}'\oplus E_{i+2k-1}')[-m_{i-1}]\xrightarrow{\scriptsize\begin{pmatrix}g_{i+k}&g_{i+2k}\end{pmatrix}}E_i\right)[1]\\
\cong cone\left(E_{i-1}\xrightarrow{\scriptsize\begin{pmatrix}h_{i+k-1}\\-h_{i+2k-1}\end{pmatrix}}E_{i+k-1}'\oplus E_{i+2k-1}'\right)[2-m_{i-1}],\\
\end{multline*}
\begin{multline*}
T_{E'}\,cone\left(E_{i-1}\xrightarrow{\scriptsize\begin{pmatrix}h_{i+k-1}\\-h_{i+2k-1}\end{pmatrix}}E_{i+k-1}'\oplus E_{i+2k-1}'\right)[2-m_{i-1}]\\
\cong cone\left(E_{i-2}'[-m_{i-2}]\xrightarrow{g_{i-1}}E_{i-1}\right)[3-m_{i-1}],\\
T_E\,cone\left(E_{i-2}'[-m_{i-2}]\xrightarrow{g_{i-1}}E_{i-1}\right)[3-m_{i-1}]\cong E_{i-2}'[4-m_{i-1}-m_{i-2}].
\end{multline*}
All of these isomorphisms can be obtained by a direct application of the octahedral axiom. Then the required conditions are proved.
\end{proof}

Thus, we get a homomorphism $\gamma$ from $\sB_\Gamma$ to the group of autoequivalences of $\sT$ defined by the equalities $\gamma(\sigma_1)=T_E$ and $\gamma(\sigma_2)=T_{E'}$ and it remains to find the kernel of $\gamma$.
Applying the octahedral axiom as it was done in the proof of Lemma \ref{actbr}, one can show that
\begin{itemize}
\item $(T_ET_{E'})^3(E_i)=E_{i-3}[4-m_{i-1}-m_{i-2}-m_{i-3}]$ and $(T_ET_{E'})^3(E_i')=E_{i-3}'[4-m_{i-1}-m_{i-2}-m_{i-3}]$ if $r=1$;
\item $(T_ET_{E'})^2(E_i)=E_{i-2}[3-m_{i-1}-m_{i-2}]$ and $(T_ET_{E'})^2(E_i')=E_{i+k-2}'[3-m_{i-1}-m_{i-2}]$ if $r=2$;
\item $(T_ET_{E'})^3(E_i)=E_{i-3}[5-m_{i-1}-m_{i-2}-m_{i-3}]$ and $(T_ET_{E'})^3(E_i')=E_{i-3}'[5-m_{i-1}-m_{i-2}-m_{i-3}]$ if $r=3$;
\end{itemize}
In particular, one has $(\gamma\Delta_\Gamma)(E)\sim E$ and $(\gamma\Delta_\Gamma)(E')\sim E'$. Note that if $\Phi$ is an autoequivalence such that $\Phi(E)\sim E$ and $\Phi(E')\sim E'$, then the permutation of $\Sph_{E,E'}$ induced by $\Phi$ is trivial. Indeed, $\Phi T_E=T_{\Phi E}\Phi=T_E\Phi$, $\Phi T_{E'}=T_{\Phi E'}\Phi=T_{E'}\Phi$, and hence $\Phi X\sim X$ if and only if $\Phi T_EX \sim T_EX$ and $\Phi T_{E'}X\sim T_{E'}X$.
Lemma \ref{actbr} and the just mentioned argument show that $\gamma$ induces an action of the group $\sB_\Gamma/Z_{\Gamma}$ on $\Sph_{E,E'}$, where $Z_\Gamma$ denote the center of the group $\sB_\Gamma$.
Our next goal is to show that this action is faithful with only two exception. To do this we will use the following lemma about groups of the form $\sB_\Gamma/Z_{\Gamma}$.

\begin{lemma}\label{ffreetobr} Suppose that $\Gamma\in\{A,B,G\}$ and $\phi:\sB_\Gamma/Z_{\Gamma}\rightarrow K$ is a homomorphism of groups. The homomorphism $\phi$ is injective if one of the following conditions holds:
\begin{enumerate}
\item $\Gamma=A$ and  the elements $\phi(\sigma_1)$ and $\phi(\sigma_2^2\sigma_1\sigma_2^{-2})$ generate a subgroup isomorphic to $F_2$;
\item $\Gamma=B$ and the elements $\phi(\sigma_1)$ and $\phi(\sigma_2\sigma_1\sigma_2^{-1}\big)$ generate a subgroup isomorphic to $F_2$;
\item $\Gamma=G$ and the elements $\phi(\sigma_1)$, $\phi(\sigma_2\sigma_1\sigma_2^{-1})$ and $\phi(\sigma_2\sigma_1\sigma_2\sigma_1\sigma_2^{-1}\sigma_1^{-1}\sigma_2^{-1})$ generate a subgroup isomorphic to $F_3$.
\end{enumerate}
\end{lemma}
\begin{proof} Let us first prove the items $(2)$ and $(3)$. Note that $\sB_B/Z_{B}\cong \mathbb{Z}*\mathbb{Z}/2\mathbb{Z}=\langle x,y\mid y^2=1\rangle$ and $\sB_G/Z_{G}\cong \mathbb{Z}*\mathbb{Z}/3\mathbb{Z}=\langle x,y\mid y^3=1\rangle$, where both isomorphisms send $\sigma_1$ to $x$ and $\sigma_2$ to $x^{-1}y$. Hence, it is enough to prove that if $H$ is a group generated by $x$ and $y$ such that $y^r=1$ and the elements $x,yxy^{-1},\dots,y^{r-1}xy^{1-r}$ generate a free group on $r$ generators, then $H$ is a free product of $\mathbb{Z}$ and $\mathbb{Z}/r\mathbb{Z}$. This can be shown, for example, in the following way. It is clear that the subgroup $F_r=\langle x,yxy^{-1},\dots,y^{r-1}xy^{1-r}\rangle\subset H$ is normal. Let us consider the action of $H$ on it by conjugation. Let us define $X\subset F_r$ as the set of elements whose reduced expressions start and finish with $x^k$ for some nonzero integer $k$ and $Y\subset F_r$ as the set of elements whose reduced expressions do not start and do not finish with $x^k$ for some nonzero integer $k$. It is easy to see that $x^{k}Yx^{-k}\subset X$ and $y^lXy^{-l}\subset Y$ for any nonzero integer $k$ and any $1\le l\le r-1$. Thus, the required assertion follows from the ping-pong lemma.

 To prove the first item, let us note first that the group generated by $\phi(\sigma_1)$ and $\phi(\sigma_2^2)$ is isomorphic to $\sB_B/Z_{B}\cong \mathbb{Z}*\mathbb{Z}/2\mathbb{Z}$ by the just proved assertion. Let us prove by induction on the length of the word in $\sigma_1$ and $\sigma_2$ representing $w\in \sB_A/Z_{A}$ that it can be presented in the form $w=(\sigma_2\sigma_1)^kw'$, where $0\le k\le 2$ and $w'\in \langle \sigma_1,\sigma_2^2\rangle$. Suppose that $w=(\sigma_2\sigma_1)^kw'$ with $w'\in \langle \sigma_1,\sigma_2^2\rangle$. We have to prove that $\sigma_1^{\pm 1}w$ and $\sigma_2^{\pm 1}w$ can be presented in the required form. If $k=0$, then $\sigma_1^{\pm 1}w=\sigma_1^{\pm 1}w'\in \langle \sigma_1,\sigma_2^2\rangle$. If $k=1$, then $\sigma_1w=\sigma_1\sigma_2\sigma_1w'=(\sigma_2\sigma_1)^2(\sigma_1^{-1}w')$ and $\sigma_1^{-1}w=\sigma_1^{-1}\sigma_2\sigma_1w'=(\sigma_2\sigma_1)^2(\sigma_2^2\sigma_1w')$. If $k=2$, then $\sigma_1w=\sigma_1\sigma_2\sigma_1\sigma_2\sigma_1w'=(\sigma_2\sigma_1)(\sigma_2^2\sigma_1w')$ and $\sigma_1^{-1}w=\sigma_1^{-1}\sigma_2\sigma_1\sigma_2\sigma_1w'=\sigma_2\sigma_1(\sigma_1w')$.
The case of the word $\sigma_2^{\pm 1}w$ can be considered in the same manner.
Let us now suppose that $(\sigma_2\sigma_1)^kw\in \Ker\phi$ for some $w\in \langle \sigma_1,\sigma_2^2\rangle$. Since $\langle \sigma_1,\sigma_2^2\rangle$ maps injectively to $K$, it follows from $\phi(w^3)=1$ that $w=1$. Thus, we have $(\sigma_2\sigma_1)^k\in \Ker\phi$, and hence $k=0$, i.e. $(\sigma_2\sigma_1)^kw=1$.
\end{proof}

Now we are ready to prove the results on the faithfulness of the action of $\sB_\Gamma/Z_{\Gamma}$ on $\Sph_{E,E'}$.

\begin{corollary}\label{Af}
If $r=1$ and $T_{E'}^{2}E\not\sim E$, then the action of $\sB_A/Z_A$ on $\Sph_{E,E'}$ induced by $\gamma$ is faithful.
\end{corollary}
\begin{proof} Due to Lemma \ref{ffreetobr}, it is enough to prove that $T_E$ and $T_{E'}^{2}T_ET_{E'}^{-2}=T_{T_{E'}^{2}E}$ generate a subgroup isomorphic to the free group on two generators in the group of permutations of $\Sph_{E,E'}$. Note that $T_{E'}^{2}E$ is an $m$-spherical sequence of length $k$, and hence due to Corollary \ref{free} it is enough to check that $a_{E,T_{E'}^{2}E}=2$.

Taking the direct summands of the triangle \eqref{twgen} with $s=0,1$ we get the triangles
$$
E_i\rightarrow T_{E'}E_i\rightarrow E_{i-1}'[1-m_{i-1}]\mbox{ and }T_{E'}E_i\rightarrow T_{E'}^{2}E_i\rightarrow E_{i-2}'[2-m_{i-1}-m_{i-2}]
$$
for all $i\in\mathbb{Z}/k\mathbb{Z}$. It is not difficult to get from the first triangle that $a_{E,T_{E'}E}=1$ and that the basis of $\oplus_{l\in\mathbb{Z}}\Hom_\sT(E,T_{E'}E_i[l])$ is formed by the morphism $E_i\rightarrow T_{E'}E_i$ from the just mentioned triangle.
Since $a_{E,E'}=1$, we get from the second triangle above that either $a_{E,T_{E'}^{2}E}=2$ or $a_{E,T_{E'}^{2}E}=0$. Suppose that $a_{E,T_{E'}^{2}E}=0$. Then combining the triangles above and using the octahedral axiom we get the commutative diagram
\begin{center}
\begin{tikzpicture}[node distance=1.4cm]
 \node(A) {$E_i$};
\node(A') [below  of=A] {$E_i$};
 \node(Ar) [right  of=A] {};
 \node(Arr) [right  of=Ar] {};
 \node(Arrr) [right  of=Arr] {};
 \node(B) [right  of=Arr] {$T_{E'}E_i$};
\node(B') [below  of=B] {$T_{E'}^{2}E_i$};
 \node(Br) [right  of=B] {};
 \node(Brr) [right  of=Br] {};
 \node(Brrr) [right  of=Brr] {};
 \node(C) [right  of=Brr] {$E_{i-1}'[1-m_{i-1}]$};
\node(C') [below  of=C] {$T_{E'}^{2}E_i\oplus E_i[1]$};
\node(B'') [above  of=B] {$E_{i-2}'[1-m_{i-1}-m_{i-2}]$};
\node(C'') [above  of=C] {$E_{i-2}'[1-m_{i-1}-m_{i-2}]$};

\draw [->,>=stealth'] (A')  to node[above]{\tiny$0$ } (B') ; 

\draw [double equal sign distance] (A)  to [out=-89, in=89] (A') ;
\draw [double equal sign distance] (B'')  to [out=1, in=179] (C'') ;
\draw [->,>=stealth'] (A)  to  (B) ; 
\draw [->,>=stealth'] (B)  to  (C) ; 
\draw [->,>=stealth'] (B')  to  (C') ; 
\draw [->,>=stealth'] (B)  to  (B') ; 
\draw [->,>=stealth'] (C)  to  (C') ; 
\draw [->,>=stealth'] (B'')  to (B) ; 
\draw [->,>=stealth'] (C'')  to node[right]{\tiny$f_{i-2}^{E'}[1-m_{i-1}]$ } (C) ; 
\end{tikzpicture}
\end{center}
whose right column is a triangle for any $i\in\mathbb{Z}/k\mathbb{Z}$. Applying $T_{E'}[m_{i-1}-1]$ to the just mentioned triangle, we get also the triangle
$$
E_{i-3}'[1-m_{i-2}-m_{i-3}]\xrightarrow{f_{i-2}^{E'}[1-m_{i-2}]} E_{i-2}'[1-m_{i-2}]\rightarrow (T_{E'}^{3}E_i\oplus T_{E'}E_i[1])[m_{i-1}-1].
$$
Then it follows from the uniqueness of a triangle containing a given morphism modulo isomorphism that $(T_{E'}^{3}E_i\oplus T_{E'}E_i[1])[m_{i-1}-1]\cong T_{E'}^{2}E_{i-1}\oplus E_{i-1}[1]$, and hence either $T_{E'}E_i[m_{i-1}]\cong T_{E'}^{2}E_{i-1}$ or $T_{E'}E_i[m_{i-1}]\cong E_{i-1}[1]$. In any case we get $T_{E'}E\sim E$, which is impossible since $a_{E,T_{E'}E}=1$. The obtained contradiction implies that $a_{E,T_{E'}^{2}E}=2$, and thus the corollary is proved.
\end{proof}

Thus, we have the required faithfulness of the action of $\sB_A/Z_A$ if $T_{E'}^{2}E\not\sim E$. The next lemma shows that this condition is satisfied except the case $m=2$, $k=3$ mentioned in the first item of Theorem \ref{main} and that in the exceptional case we really have an action of the group $S_3^\mathbb{Z}$.
Moreover, in the next section we will show that this situation really can occur.

\begin{lemma}\label{exceptA}
If $r=1$ and $T_{E'}^{2}E\sim E$, then $(m,k)=(2,3)$ and the group generated by $T_E$ and $T_{E'}$ is isomorphic to $S_3^\mathbb{Z}$.
\end{lemma}
\begin{proof} Suppose that $T_{E'}^{2}E\sim E$. The existence of the triangle $T_{E'}E_i\rightarrow T_{E'}^{2}E_i\rightarrow E_{i-2}'[2-m_{i-1}-m_{i-2}]$ implies that the basis of $\oplus_{l\in\mathbb{Z}}\Hom_\sT(E[l],T_{E'}^{2}E_i)$ is formed by the morphisms $E_i\rightarrow T_{E'}^{2}E_i$ and $E_{i-2}[2-m_{i-1}-m_{i-2}]\rightarrow T_{E'}^{2}E_i$ fist of which can be factored as $E_i\rightarrow T_{E'}E_i\rightarrow T_{E'}^{2}E_i$. Then the morphism $E_i\rightarrow T_{E'}^{2}E_i$ cannot be invertible, and hence $T_{E'}^{2}E_i\cong E_{i-2}[2-m_{i-1}-m_{i-2}]$. Using the same triangle, one can see that the basis of $\oplus_{l\in\mathbb{Z}}\Hom_\sT(T_{E'}^{2}E_i,E[l])$ is formed by the morphisms $T_{E'}^{2}E_i\rightarrow E_{i+1}[m_i]$ and $T_{E'}^{2}E_i\rightarrow E_{i-1}[2-m_{i-1}]$ second of which is annihilated by the morphism $T_{E'}E_i\rightarrow T_{E'}^{2}E_i$. Then we get also $T_{E'}^{2}E_i\cong E_{i+1}[m_i]$. The condition $E_{i-2}[2-m_{i-1}-m_{i-2}]\cong E_{i+1}[m_i]$ can be satisfied only if $k\mid 3$ and $m_{i-1}+m_{i-2}+m_i=2$, i.e. only if $(m,k)=(2,3)$.

Let us now prove that the homomorphism $\phi:S_3^\mathbb{Z}\rightarrow \Aut(\sT)$ sending $\sigma_1$ to $T_E$ and $\sigma_2$ to $T_{E'}$ is well defined. Since we already know that $T_ET_{E'}T_E=T_{E'}T_ET_{E'}$, it is enough to show that $T_E^{2}=T_{E'}^{2}$. Since $T_{E'}^{2}E\sim E$, we have $T_{E'}^{2}T_E=T_{T_{E'}^{2}E}T_{E'}^{2}=T_ET_{E'}^{2}$. Multiplying this equality by $T_{E'}$ on the left, we get $T_{E'}^{3}T_E=T_{E'}T_ET_{E'}^{2}=T_ET_{E'}T_ET_{E'}=T_E^{2}T_{E'}T_E$, and hence $T_{E'}^{2}=T_E^{2}$.

Since $T_{E'}^{2}E\sim E$ and $T_{E'}^{2}E'\sim E'$, we get an action of the symmetric group $S_3$ on $\Sph_{E,E'}$. Note that $T_{E'}$ fixes $E'$ and interchanges $E$ and $T_{E'}E$ while $T_E$ fixes $E$ and interchanges $E'$ and $T_EE'\sim T_E^{2}T_{E'}E\sim T_{E'}E$. Then the action of $S_3$ on $\Sph_{E,E'}$ is faithful, and hence the kernel of $\phi$ is contained in the subgroup of $S_3^\mathbb{Z}$ generated by $\sigma_2^{2}$ (note that this subgroup is the center of $S_3^\mathbb{Z}$). Thus, it remains to show that $\phi(\sigma_2^{2t})=T_{E'}^{2t}\not\cong\Id_\sT$ for any nonzero integer $t$.
But it follows from \eqref{twistonobj} that $T_{E'}^{6t}E_i'=E_{i}'[2t]\not\cong E_i'$, and hence the proof of the lemma is finished.
\end{proof}

Let us now consider the case $\Gamma=B$.

\begin{corollary}\label{Bf}
If $r=2$ and $T_{E}E'\not\sim E'$, then the action of $\sB_B/Z_B$ on $\Sph_{E,E'}$ induced by $\gamma$ is faithful.
\end{corollary}
\begin{proof} Due to Lemma \ref{ffreetobr}, it is enough to prove that $T_{E'}$ and $T_{E}T_{E'}T_{E}^{-1}=T_{T_{E}E'}$ generate a subgroup isomorphic to the free group on two generators in the group of permutations of $\Sph_{E,E'}$. Due to Corollary \ref{free}, it is enough to check that $a_{E',T_{E}E'}=2$. This follows from the existence of the triangle
$$
E_i'\rightarrow T_{E}E_i'\rightarrow E_{i}[1]
$$
which is simply the triangle \eqref{tw1} adopted to the case under consideration.
\end{proof}

Thus, we have the required faithfulness of the action of $\sB_B/Z_B$ if $T_{E}E'\not\sim E'$. The next lemma shows that this condition is satisfied except the case $m=1$, $k=2$ mentioned in the second item of Theorem \ref{main} and that in the exceptional case we have an action of the group $(\mathbb{Z}\times \mathbb{Z})/(2t,-2t)$ for some integer $t$.
Moreover, in the next section we will show that this situation really can occur.

\begin{lemma}\label{exceptB}
If $r=2$ and $T_{E}E'\sim E'$, then $(m,k)=(1,2)$ and the group generated by $T_E$ and $T_{E'}$ is isomorphic to $(\mathbb{Z}\times \mathbb{Z})/(2t,-2t)$ for some integer $t$.
\end{lemma}
\begin{proof} Suppose that $T_{E}E'\sim E'$. The existence of the triangle $E_i'\rightarrow T_{E}E_i'\rightarrow E_{i}[1]$ implies that the basis of $\oplus_{l\in\mathbb{Z}}\Hom_\sT(E'[l],T_{E}E_i')$ is formed by the morphisms $E_i'\rightarrow T_{E}E_i'$ and $E_{i+k-1}'[1-m_{i-1}]\rightarrow T_{E}E_i'$ fist of which is annihilated by the nonzero morphism $T_{E}E_i'\rightarrow E_{i}[1]$, and hence cannot be isomorphism. Then $T_{E}E_i'\cong E_{i+k-1}'[1-m_{i-1}]$. Using the same triangle, one can see that the basis of $\oplus_{l\in\mathbb{Z}}\Hom_\sT(T_{E}E_i',E'[l])$ is formed by the morphisms $T_{E}E_i'\rightarrow E_{i+1}'[m_i]$ and $T_{E}E_i'\rightarrow E_{i}'[1]$ second of which is annihilated by the morphism $E_i'\rightarrow T_{E}E_i'$, and hence cannot be isomorphism. Then we get also $T_{E}E_i'\cong E_{i+1}'[m_i]$. The condition $E_{i+k-1}'[1-m_{i-1}]\cong E_{i+1}[m_i]$ can be satisfied only if $2k\mid k-2$ and $m_{i-1}+m_i=1$, i.e. only if $(m,k)=(1,2)$.

Since $T_{E}T_{E'}\cong T_{T_{E}E'}T_{E}\cong T_{E'}T_{E}$, the group generated by $T_E$ and $T_{E'}$ is isomorphic to a factor group of $\mathbb{Z}\times\mathbb{Z}$. By \eqref{twistonobj} we have $T_{E'}E_i'=E_{i-1}'[1-m_{i-1}]$ for any $i\in\mathbb{Z}/4\mathbb{Z}$. On the other hand, we have shown that $T_{E}E_i'\cong E_{i+1}'[1-m_{i-1}]$. Suppose that the element $(a,b)\in \mathbb{Z}\times\mathbb{Z}$ lies in the kernel of the homomorphism from $\mathbb{Z}\times\mathbb{Z}$ to $\Aut(\sT)$ sending $(1,0)$ to $T_E$ and $(0,1)$ to $T_{E'}$. We may assume for convenience that $a+b\ge 0$. Then we have
$E_i'\cong T_E^aT_{E'}^bE_i'=E_{i+a-b}'[a+b-\sum\limits_{s=1}^{a+b}m_{i-s}]$, and hence $4\mid a-b$ and $a+b=\frac{a+b}{2}$, i.e. $b=-a$ and $2\mid a$. Thus,  the group generated by $T_E$ and $T_{E'}$ is isomorphic to $(\mathbb{Z}\times \mathbb{Z})/(2t,-2t)$ for some integer $t$.
\end{proof}

It remains to prove the faithfulness for $\Gamma=G$.

\begin{corollary}\label{Gf}
If $r=3$, then the action of $\sB_G/Z_G$ on $\Sph_{E,E'}$ induced by $\gamma$ is faithful.
\end{corollary}
\begin{proof} Due to Lemma \ref{ffreetobr}, it is enough to prove that $T_E$, $T_{E'}T_ET_{E'}^{-1}=T_{T_{E'}E}$ and $(T_{E'}T_ET_{E'})T_E(T_{E'}T_ET_{E'})^{-1}=T_{T_{E'}T_ET_{E'}E}$ generate a subgroup isomorphic to the free group on three generators in the group of permutations of $\Sph_{E,E'}$.
Due to Corollary \ref{free3}, it is enough to check that $a_{E,T_{E'}E}=a_{E,(T_{E'}T_ET_{E'})E}=a_{T_{E'}E,(T_{E'}T_ET_{E'})E}=3$ and there exists a triangle of the form $E\rightarrow F\rightarrow T_{E'}E$ with $F\sim T_{E'}T_ET_{E'}E$. Direct calculations using the octahedral axiom show that
$$
T_{E'}E_i\cong cone\left((E_{i-1}'\oplus E_{i+k-1}'\oplus E_{i+2k-1}')[-m_{i-1}]\xrightarrow{\scriptsize \begin{pmatrix}g_{i}&g_{i+k}&g_{i+2k}\end{pmatrix}}E_i\right),\\
$$
\begin{multline*}
T_ET_{E'}E_i[m_{i-1}-1]\\
\cong cone\left(E_{i-1}^2\xrightarrow{\scriptsize\begin{pmatrix}h_{i-1}&0\\-h_{i+k-1}&h_{i+k-1}\\0&-h_{i+2k-1}\end{pmatrix}}(E_{i-1}'\oplus E_{i+k-1}'\oplus E_{i+2k-1}')\right).
\end{multline*}
and 
\begin{multline*}
T_{E'}T_ET_{E'}E_i[m_{i-1}-2]\\
\cong cone\left((E_{i-2}'\oplus E_{i+k-2}'\oplus E_{i+2k-2}')[-m_{i-2}]\xrightarrow{\scriptsize\begin{pmatrix}0&g_{i+k-1}&g_{i+2k-1}\\g_{i-1}&g_{i+k-1}&0\end{pmatrix}}E_{i-1}^2\right).
\end{multline*}
Now the equalities $a_{E,T_{E'}E}=a_{E,(T_{E'}T_ET_{E'})E}=3$ and $a_{T_{E'}E,(T_{E'}T_ET_{E'})E}=a_{E,(T_ET_{E'})E}=3$ can be easily verified. Let us choose some $\alpha,\beta\in\kk^*$ such that $\alpha+\beta\not=0$. Applying the octahedral axiom to the composition
$$\begin{pmatrix}\alpha\Id_{E_i}&\beta\Id_{E_i}\end{pmatrix}\circ \begin{pmatrix}0&g_{i+k}&g_{i+2k}\\g_{i}&g_{i+k}&0\end{pmatrix}=\begin{pmatrix}\beta g_{i}&(\alpha+\beta)g_{i+k}&\alpha g_{i+2k}\end{pmatrix}$$
and noting that $cone\begin{pmatrix}\beta g_{i}&(\alpha+\beta)g_{i+k}&\alpha g_{i+2k}\end{pmatrix}\cong cone\begin{pmatrix}g_{i}&g_{i+k}& g_{i+2k}\end{pmatrix}$, we get the triangles
$$
E_i\rightarrow T_{E'}T_ET_{E'}E_{i+1}[m_{i}-2]\rightarrow T_{E'}E_i
$$
for all $i\in\mathbb{Z}/k\mathbb{Z}$. Taking the direct sum of these triangles, we get the required triangle $E\rightarrow F\rightarrow T_{E'}E$ with $F\sim T_{E'}T_ET_{E'}E$.
\end{proof}

Let us now deduce Theorem \ref{main} from our results.

\begin{proof}[Proof of Theorem \ref{main}] Since item \ref{last} is already proved, we may assume that we are in the settings of one of the items \ref{a2}, \ref{b2} and \ref{g2}.
Due to Corollaries \ref{Af}, \ref{Bf}, \ref{Gf} and Lemmas \ref{exceptA} and \ref{exceptB}, if the action of $\sB_\Gamma/Z_{\Gamma}$ on $\Sph_{E,E'}$ is not faithful, then either $r=1$, $(m,k)=(2,3)$ and the group generated by $T_E$ and $T_{E'}$ is isomorphic to $S_3^{\mathbb{Z}}$ or $r=2$, $(m,k)=(1,2)$ and the group generated by $T_E$ and $T_{E'}$ is isomorphic to $(\mathbb{Z}\times \mathbb{Z})/(2t,-2t)$ for some integer $t$, i.e. some condition of Theorem \ref{main} is satisfied.

If the action of $\sB_\Gamma/Z_{\Gamma}$ on $\Sph_{E,E'}$ is faithful, then the group generated by $T_E$ and $T_{E'}$ is isomorphic to $\sB_{\Gamma}/\langle\Delta_\Gamma^t\rangle$ for some nonnegative integer $t$. Let us consider all values of $\Gamma$ separately.
\begin{itemize}
\item Suppose that $\Gamma=A$. Then $t$ has to satisfy the condition $(T_ET_{E'})^{3t}\cong \Id_{\sT}$. Then we have $E_i'\cong (T_ET_{E'})^{3t}E_i'=E_{i-3t}'\left[4t-\sum\limits_{s=1}^{3t}m_{i-s}\right]$. Then we have $k\mid 3t$, i.e. $\frac{k}{gcd(k,3)}\mid t$. Now we have $(T_ET_{E'})^{3a\frac{k}{gcd(k,3)}}E_i'=E_{i'}\left[\frac{(4k-3m)a}{gcd(k,3)}\right]$, i.e. $t$ can be nonzero only if $3m=4k$.
\item Suppose that $\Gamma=B$. Then $t$ has to satisfy the condition $(T_ET_{E'})^{2t}\cong \Id_{\sT}$. Then we have $E_i'\cong (T_ET_{E'})^{2t}E_i'=E_{i+(k-2)t}'\left[3t-\sum\limits_{s=1}^{2t}m_{i-s}\right]$. Then we have $2k\mid (k-2)t$, i.e. $\frac{2k}{gcd(k-2,4)}\mid t$. Now we have $(T_ET_{E'})^{2a\frac{2k}{gcd(k-2,4)}}E_i'=E_{i'}\left[	\frac{(3k-2m)2a}{gcd(k-2,4)}\right]$, i.e. $t$ can be nonzero only if $2m=3k$.
\item Suppose that $\Gamma=G$.  Then $t$ has to satisfy the condition $(T_ET_{E'})^{3t}\cong \Id_{\sT}$. Then we have $E_i'\cong (T_ET_{E'})^{3t}E_i'=E_{i-3t}'\left[5t-\sum\limits_{s=1}^{3t}m_{i-s}\right]$, and hence $k\mid t$. Now we have $(T_ET_{E'})^{3ak}E_i'=E_{i'}[(5k-3m)a]$, i.e. $t$ can be nonzero only if $3m=5k$.
\end{itemize}
\end{proof}

\section{Derived categories of hereditary algebras and non faithful actions of braid groups}

In this section we will consider spherical sequences in the bounded derived categories of hereditary algebras of types $A_3$ and $D_4$. Note that due to \cite{Hap} these categories are equivalent to the stable categories of associated repetitive algebras.
In fact there are several hereditary algebras and corresponding to them repetitive algebras of types $A_3$ and $D_4$, but the bounded derived categories of hereditary algebras of the same type are equivalent.

The bounded derived category of a hereditary algebra of type $D_4$ that we will denote by $\Db(D_4)$ can be described in the following way (see \cite{Hap}). Let us consider the quiver $\mathbb{Z}D_4$ with the vertex set $\{0,1,2,3\}\times\mathbb{Z}$ and the arrows $(0,s)\xrightarrow{\alpha_{r,s}} (r,s)$ and $(r,s)\xrightarrow{\beta_{r,s}}(0,s+1)$ for $r\in\{1,2,3\}$, $s\in\mathbb{Z}$. Let us consider the ideal $I_{D_4}$ of $\kk\mathbb{Z}D_4$ generated by linear combinations of paths $\alpha_{r,s}\beta_{r,s-1}$ and $\beta_{1,s}\alpha_{1,s}+\beta_{2,s}\alpha_{2,s}+\beta_{3,s}\alpha_{3,s}$ for $r\in\{1,2,3\}$, $s\in\mathbb{Z}$.
The ideal $I_{D_4}$ is called the {\it mesh ideal}. The category whose objects are the vertices of $\mathbb{Z}D_4$, morphisms from $e_1$ to $e_2$ are elements of $e_2(\kk\mathbb{Z}D_4/I_{D_4})e_1$ and the composition is induced by the multiplication in $\kk\mathbb{Z}D_4/I_{D_4}$ is called the {\it mesh category} of $D_4$. It is denoted by $\kk(D_4)$. The subcategory of $\Db(D_4)$ formed by indecomposable objects is equivalent to $\kk(D_4)$. The category $\Db(D_4)$ has a Serre functor $\sS$ and the shift functor $[1]$. On indecomposable objects these functors are defined by the equalities $\sS(r,s)=(r,s+2)$ and $(r,s)[1]=(r,s+3)$.
Let us define $E_i=(1,-i)$ and $E_i'=(2,1-i)$ for $i\in\{0,1,2\}$. Then $E$ and $E'$ are sherical sequences of length $3$ and sphericity $2$ with $m_0=m_1=m_0'=m_1'=1$ and $m_2=m_2'=0$. Moreover, they are adjusted in such a way that $\Hom_{\Db(D_4)}(E_i,E_j'[l])$ is one dimensional if $i=j$ and $l=0$ and equals zero otherwise.

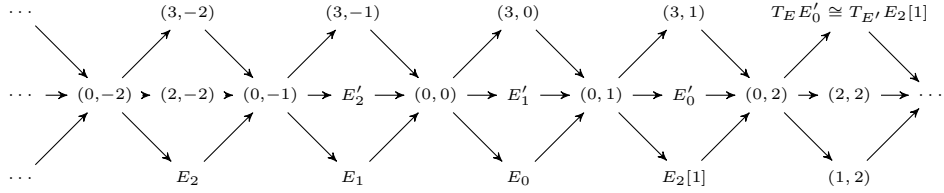
\begin{center}
\begin{tikzpicture}[node distance=1.1cm]
 \node(d3) {\tiny$\cdots$};
 \node(d2) [below  of=d3] {\tiny$\cdots$};
 \node(d1) [below  of=d2] {\tiny$\cdots$};
 \node(m20) [right  of=d2] {\tiny$(0,\!-2)$};
 \node(m22) [right  of=m20] {\tiny$(2,\!-2)$};
 \node(m21) [below  of=m22] {\tiny$E_2$};
 \node(m23) [above  of=m22] {\tiny$(3,\!-2)$};
 \node(m10) [right  of=m22] {\tiny$(0,\!-1)$};
 \node(m12) [right  of=m10] {\tiny$E_2'$};
 \node(m11) [below  of=m12] {\tiny$E_1$};
 \node(m13) [above  of=m12] {\tiny$(3,\!-1)$};
 \node(00) [right  of=m12] {\tiny$(0,0)$};
 \node(02) [right  of=00] {\tiny$E_1'$};
 \node(01) [below  of=02] {\tiny$E_0$};
 \node(03) [above  of=02] {\tiny$(3,0)$};
 \node(10) [right  of=02] {\tiny$(0,1)$};
 \node(12) [right  of=10] {\tiny$E_0'$};
 \node(11) [below  of=12] {\tiny$E_2[1]$};
 \node(13) [above  of=12] {\tiny$(3,1)$};
 \node(20) [right  of=12] {\tiny$(0,2)$};
 \node(22) [right  of=20] {\tiny$(2,2)$};
 \node(21) [below  of=22] {\tiny$(1,2)$};
 \node(23) [above  of=22] {\tiny$T_{E}E_0'\cong T_{E'}E_2[1]$};
 \node(f) [right  of=22] {\tiny$\cdots$};
\draw [->,>=stealth'] (d1)  to  (m20) ; 
\draw [->,>=stealth'] (d2)  to  (m20) ; 
\draw [->,>=stealth'] (d3)  to  (m20) ; 

\draw [->,>=stealth']  (m20) to (m21); 
\draw [->,>=stealth']  (m20)  to (m22) ; 
\draw [->,>=stealth']  (m20)  to (m23); 
\draw [->,>=stealth']  (m21) to (m10); 
\draw [->,>=stealth']  (m22) to (m10); 
\draw [->,>=stealth']  (m23) to (m10);

\draw [->,>=stealth']  (m10) to (m11); 
\draw [->,>=stealth']  (m10)  to (m12) ; 
\draw [->,>=stealth']  (m10)  to (m13); 
\draw [->,>=stealth']  (m11) to (00); 
\draw [->,>=stealth']  (m12) to (00); 
\draw [->,>=stealth']  (m13) to (00);

\draw [->,>=stealth']  (00) to (01); 
\draw [->,>=stealth']  (00)  to (02); 
\draw [->,>=stealth']  (00)  to (03); 
\draw [->,>=stealth']  (01) to (10); 
\draw [->,>=stealth']  (02) to (10); 
\draw [->,>=stealth']  (03) to (10);

\draw [->,>=stealth']  (10) to (11); 
\draw [->,>=stealth']  (10)  to (12) ; 
\draw [->,>=stealth']  (10)  to (13); 
\draw [->,>=stealth']  (11) to (20); 
\draw [->,>=stealth']  (12) to (20); 
\draw [->,>=stealth']  (13) to (20);

\draw [->,>=stealth']  (20) to (21); 
\draw [->,>=stealth']  (20)  to (22) ; 
\draw [->,>=stealth']  (20)  to (23); 
\draw [->,>=stealth']  (21) to (f); 
\draw [->,>=stealth']  (22) to (f); 
\draw [->,>=stealth']  (23) to (f);
\end{tikzpicture}
\end{center}

Now, applying the octahedral axiom and using the AR triangles $E_0\xrightarrow{\beta_{1,0}} (0,1)\xrightarrow{\alpha_{1,1}} E_2[1]$ and $E_1'\xrightarrow{\beta_{2,0}} (0,1)\xrightarrow{\alpha_{2,1}} E_0'$ we get the commutative diagram
\begin{center}
\begin{tikzpicture}[node distance=1cm]
 \node(A) {$E_0$};
\node(A') [below  of=A] {$E_0$};
 \node(Ar) [right  of=A] {};
 \node(Arr) [right  of=Ar] {};
 \node(Arrr) [right  of=Arr] {};
 \node(B) [right  of=Ar] {$(0,1)$};
\node(B') [below  of=B] {$E_0'$};
 \node(Br) [right  of=B] {};
 \node(Brr) [right  of=Br] {};
 \node(Brrr) [right  of=Brr] {};
 \node(C) [right  of=Br] {$E_2[1]$};
\node(C') [below  of=C] {$X$};
\node(B'') [above  of=B] {$E_1'$};
\node(C'') [above  of=C] {$E_1'$};

\draw [->,>=stealth'] (A')  to node[above]{\tiny$\alpha_{2,1}\beta_{1,0}$ } (B') ; 

\draw [double equal sign distance] (A)  to [out=-89, in=89] (A') ;
\draw [double equal sign distance] (B'')  to [out=1, in=179] (C'') ;
\draw [->,>=stealth'] (A)  to node[above]{\tiny$\beta_{1,0}$ }  (B) ; 
\draw [->,>=stealth'] (B)  to  node[above]{\tiny$\alpha_{0,1}$ } (C) ; 
\draw [->,>=stealth'] (B')  to  (C') ; 
\draw [->,>=stealth'] (B)  to node[right]{\tiny$\alpha_{2,1}$ }  (B') ; 
\draw [->,>=stealth'] (C)  to  (C') ; 
\draw [->,>=stealth'] (B'')  to node[right]{\tiny$\beta_{2,0}$ } (B) ; 
\draw [->,>=stealth'] (C'')  to node[right]{\tiny$\alpha_{0,1}\beta_{2,0}$ } (C) ; 
\end{tikzpicture}
\end{center}
whose rows and columns are triangles. The right vertical and lower horizontal triangles give the isomorphisms $T_EE_0'\cong X\cong T_{E'}E_2[1]$. Thus, $T_{E'}E\sim T_EE'$, and hence $T_{E'}^2E\sim T_{E'}T_EE'\sim E$. In fact, one can show that $T_{E}E_0'\cong (3,2)$ and that the action of $T_E$ and $T_{E'}$ on the vertices of $\mathbb{Z}D_4$ is defined by the equalities
\begin{multline*}
T_E(0,i)=T_{E'}(0,i)=(0,i+1),\,T_E(1,i)=(1,i+1),\,T_{E'}(1,i)=(3,i+1),\\T_E(2,i)=(3,i+1),\,T_{E'}(2,i)=(2,i+1),\,T_E(3,i)=(2,i+1),\,T_{E'}(3,i)=(1,i+1).
\end{multline*}
Since we get the condition $T_{E'}^2E\sim E$, the subgroup of $\Aut\big(\Db(D_4)\big)$ generated by $T_E$ and $T_{E'}$ is isomorphic to $S_3^{\mathbb{Z}}$ by Lemma \ref{exceptA}.

In analogous way we describe the bounded derived category $\Db(A_3)$ of a hereditary algebra of type $A_3$. The quiver $\mathbb{Z}A_3$ has the vertex set $\{-1,0,1\}\times\mathbb{Z}$ and the arrows $(0,s)\xrightarrow{\alpha_{r,s}} (r,s)$ and $(r,s)\xrightarrow{\beta_{r,s}}(0,s+1)$ for $r=\pm 1$ and $s\in\mathbb{Z}$. The mesh ideal $I_{A_3}$ of $\kk\mathbb{Z}A_3$ is generated by linear combinations of paths $\alpha_{r,s}\beta_{r,s-1}$ and $\beta_{-1,s}\alpha_{-1,s}+\beta_{1,s}\alpha_{1,s}$ for $r=\pm 1$ and $s\in\mathbb{Z}$.
Then we get the corresponding mesh category $\kk(A_3)$ equivalent to the subcategory of $\Db(A_3)$ formed by indecomposable objects.
The category $\Db(A_3)$ has the Serre functor $\sS$ and the shift functor $[1]$ that are defined on indecomposable objects by the equalities $\sS(r,s)=(-r,s+1)$ and $(r,s)[1]=(-r,s+2)$.
Let us define $E_i=(0,-i)$, $E_i'=(1,-i)$ and $E_{i+2}'=(-1,-i)$ for $i\in\{0,1\}$. Then $E$ is a sherical sequence of length $2$ and sphericity $1$ and $E'$ is a sherical sequence of length $4$ and sphericity $2$ with $m_0=m_0'=m_2'=1$ and $m_1=m_1'=m_3'=0$. Moreover, they are adjusted in such a way that $\Hom_{\Db(A_3)}(E_i,E_j'[l])$ is one dimensional if $2\mid j-i$ and $l=0$ and equals zero otherwise. Since $E_1'\xrightarrow{\beta_{1,-1}}E_0\xrightarrow{\alpha_{1,0}}E_0'$ is an AR triangle, we have $T_EE_0'\cong E_1'[1]$, and hence $T_EE'\sim E'$. Then the subgroup of $\Aut\big(\Db(A_3)\big)$ generated by $T_E$ and $T_{E'}$ is isomorphic to $(\mathbb{Z}\times \mathbb{Z})/(2t,-2t)$ for some integer $t$ by Lemma \ref{exceptB}. In fact, one can show that $T_E^2\cong T_{E'}^2$ in this case, and hence the subgroup of $\Aut\big(\Db(A_3)\big)$ generated by $T_E$ and $T_{E'}$ is isomorphic to $(\mathbb{Z}\times \mathbb{Z})/(2,-2)\cong \mathbb{Z}\times \mathbb{Z}/2\mathbb{Z}$.

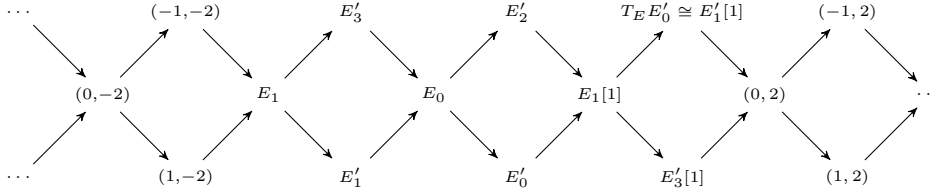
\begin{center}
\begin{tikzpicture}[node distance=1.1cm]
 \node(d3) {\tiny$\cdots$};
 \node(d2) [below  of=d3] {};
 \node(d1) [below  of=d2] {\tiny$\cdots$};
 \node(m20) [right  of=d2] {\tiny$(0,\!-2)$};
 \node(m22) [right  of=m20] {};
 \node(m21) [below  of=m22] {\tiny$(1,\!-2)$};
 \node(m23) [above  of=m22] {\tiny$(-1,\!-2)$};
 \node(m10) [right  of=m22] {\tiny$E_1$};
 \node(m12) [right  of=m10] {};
 \node(m11) [below  of=m12] {\tiny$E_1'$};
 \node(m13) [above  of=m12] {\tiny$E_3'$};
 \node(00) [right  of=m12] {\tiny$E_0$};
 \node(02) [right  of=00] {};
 \node(01) [below  of=02] {\tiny$E_0'$};
 \node(03) [above  of=02] {\tiny$E_2'$};
 \node(10) [right  of=02] {\tiny$E_1[1]$};
 \node(12) [right  of=10] {};
 \node(11) [below  of=12] {\tiny$E_3'[1]$};
 \node(13) [above  of=12] {\tiny$T_EE_0'\cong E_1'[1]$};
 \node(20) [right  of=12] {\tiny$(0,2)$};
 \node(22) [right  of=20] {};
 \node(21) [below  of=22] {\tiny$(1,2)$};
 \node(23) [above  of=22] {\tiny$(-1,2)$};
 \node(f) [right  of=22] {\tiny$\cdots$};
\draw [->,>=stealth'] (d1)  to  (m20) ; 
\draw [->,>=stealth'] (d3)  to  (m20) ; 

\draw [->,>=stealth']  (m20) to (m21); 
\draw [->,>=stealth']  (m20)  to (m23); 
\draw [->,>=stealth']  (m21) to (m10); 
\draw [->,>=stealth']  (m23) to (m10);

\draw [->,>=stealth']  (m10) to (m11); 
\draw [->,>=stealth']  (m10)  to (m13); 
\draw [->,>=stealth']  (m11) to (00); 
\draw [->,>=stealth']  (m13) to (00);

\draw [->,>=stealth']  (00) to (01); 
\draw [->,>=stealth']  (00)  to (03); 
\draw [->,>=stealth']  (01) to (10); 
\draw [->,>=stealth']  (03) to (10);

\draw [->,>=stealth']  (10) to (11); 
\draw [->,>=stealth']  (10)  to (13); 
\draw [->,>=stealth']  (11) to (20); 
\draw [->,>=stealth']  (13) to (20);

\draw [->,>=stealth']  (20) to (21); 
\draw [->,>=stealth']  (20)  to (23); 
\draw [->,>=stealth']  (21) to (f); 
\draw [->,>=stealth']  (23) to (f);
\end{tikzpicture}
\end{center}

\section{Application to derived Picard groups}

All modules in this paper are right module. All algebras considered in this section are finite dimensional. A complex $X$ is by definition a $\mathbb{Z}$-graded module with a differential $d$ of degree $1$, i.e. e sequence $\cdots\rightarrow X_i\xrightarrow{d_{i+1}}X_{i+1}\xrightarrow{d_{i+2}}X_{i+2}\rightarrow\cdots$ such that $d_{i+1}d_i=0$ for any integer $i$. The complex $X$ is concentrated in degrees from $l$ to $r$ if $X_i=0$ for $i<l$ and $i>r$. If at the same time $X_l,X_r\not=0$, then we say that $X$ has length $r-l+1$.
For an algebra $\Lambda$, we denote by $\Cb_{\Lambda}$, $\Kbp_{\Lambda}$ and $\Db_{\Lambda}$ the category of bounded complexes of finitely generated $\Lambda$-modules, the  bounded homotopy category of  finitely generated projective $\Lambda$-modules and the bounded derived category of  finitely generated $\Lambda$-modules respectively. We denote by $J_\Lambda$ the Jacobson radical of $\Lambda$. Let us recall also that any object of $\Kbp_{\Lambda}$ can be represented by a unique modulo isomorphism in the category $\Cb_{\Lambda}$ complex $(X,d)$ such that $\Im d\subset XJ_\Lambda$. Such a complex $X$ is called {\it a radical complex}. If the complexes $X$ and $Y$ represent the same element of $\Kbp_{\Lambda}$ and $X$ is radical, then $X$ is called the {\it radical representative} of $Y$. We denote by $L(Y)$ the length of the radical representative of $Y$.

Let us recall that $X\in\Kbp_{\Lambda}$ is called {\it pretilting complex} if $\Hom_{\Kbp_{\Lambda}}(X,X[i])=0$ for any nonzero integer $i$.
If $X$ is pretilting and additionally the smallest full triangulated subcategory of $\Kbp_{\Lambda}$ which contains $X$ and is closed under direct summands coincides with $\Kbp_{\Lambda}$, then $X$ is called {\it tilting}.

It was proved in \cite{Ric} that the algebras $\Lambda$ and $\Gamma$ are derived equivalent if and only if there exists a tilting complex $X\in\Kbp_{\Lambda}$ such that $\End_{\Kbp_{\Lambda}}(X)$ is isomorphic to $\Gamma$ as a $\kk$-algebra.
In the same paper it is explained how to construct an equivalence from $\Db_{\Gamma}$ to $\Db_{\Lambda}$ sending $\Gamma$ to $X$ using the tilting complex $X$ and an algebra isomorphism $\Gamma\cong\End_{\Kbp_{\Lambda}}(X)$.
One can look also in \cite{VZ1,VZ2} how to construct an equivalence from $\Kbp_{\Gamma}$ to $\Kbp_{\Lambda}$ using the same data. In the current paper we will use the fact that if $U=(U_i\rightarrow\dots\rightarrow U_j)$ is an object of $\Kbp_{\Gamma}$, then the corresponding equivalence sends $U$ to a totalization of a bicomplex whose $k$-th column is the images of $U_k$ under this equivalence, while the image of $U_k$ can be calculated using the fact that $U_k$ is a direct sum of direct summands of $\Gamma$.
Equivalences that can be constructed using the just mentioned algorithm are called standard equivalences. Standard equivalences from $\Kbp_{\Lambda}$ to itself considered modulo natural isomorphisms constitute a group under composition which is called the {\it derived Picard group} of $\Lambda$ and is denoted by $\TrPic(\Lambda)$ (see \cite{VZ1}). This group was first introduced in \cite{RZ,Ye1} as a group of tilting complexes of $\Lambda$-bimodules under the operation of derived tensor product.

Let us recall that, for a finite dimensional algebra $\Lambda$, the Picard group $\Pic(\Lambda)$ is the group of autoequivalences of the category of $\Lambda$-modules modulo natural isomorphisms. If $\Lambda$ is basic, then this group is isomorphic to the group of outer automorphisms $\Out(\Lambda)=\Aut(\Lambda)/\Inn(\Lambda)$. Here $\Aut(\Lambda)$ is the group of automorphisms of $\Lambda$ and $\Inn(\Lambda)$ is the group of inner automorphisms of $\Lambda$. This isomorphism is induced by the map $\Aut(\Lambda)\rightarrow \Pic(\Lambda)$ that sends an automorphism $\theta:\Lambda\rightarrow\Lambda$ to the autoequivalence $-\otimes_{\Lambda}\Lambda_{\theta^{-1}}$. Here $\Lambda_{\theta^{-1}}$ is  the bimodule coinciding with $\Lambda$ as a left module and having the right multiplication $*$ by elements of $\Lambda$ defined by the equality $x*a=x\theta^{-1}(a)$, where the multiplication on the right side is the original multiplication of $\Lambda$. The group $\Pic(\Lambda)$ is a subgroup of $\TrPic(\Lambda)$ in a natural way. Moreover, an element of $\TrPic(\Lambda)$ belongs to $\Pic(\Lambda)$ if and only if the radical representative of the corresponding tilting complex is concentrated in degree zero. In fact, the autoequivalence $-\otimes_{\Lambda}\Lambda_{\theta^{-1}}$ can be defined by the tilting complex $\Lambda$ and the isomorphism from $\Lambda$ to $\End_{\Lambda}(\Lambda)$ that sends $x\in\Lambda$ to the left multiplication by $\theta(x)$.
Let us recall also that $\Pic_0(\Lambda)$ is the subgroup of $\Pic(\Lambda)$ fixing all $\Lambda$-modules. Unlike $\Pic(\Lambda)$, the group $\Pic_0(\Lambda)$ is preserved by standard derived equivalences.

Let us now introduce the series of algebras $\Lambda_k$ ($k\ge 1)$. We define the quiver $Q_k$. Its vertex set is $\mathbb{Z}/k\mathbb{Z}\cup \mathbb{Z}/3k\mathbb{Z}$. For an integer $i$ we will denote by $i$ its class in $\mathbb{Z}/3k\mathbb{Z}$ and by $\overline{i}$ its class in $\mathbb{Z}/k\mathbb{Z}$. 
The arrows of $Q_k$ are $\alpha_i:\overline{i}\rightarrow i$ and $\beta_i:i\rightarrow \overline{i+1}$ ($i\in \mathbb{Z}/3k\mathbb{Z}$).

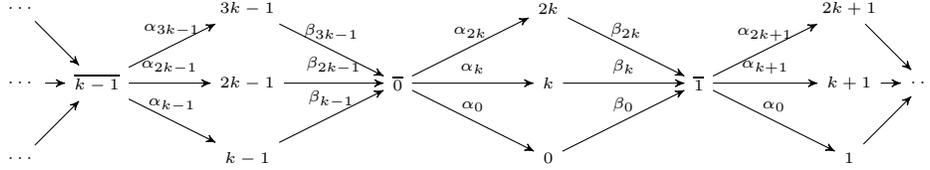
\begin{center}
\begin{tikzpicture}[node distance=1.0cm]
 \node(d3) {\tiny$\cdots$};
 \node(d2) [below  of=d3] {\tiny$\cdots$};
 \node(d1) [below  of=d2] {\tiny$\cdots$};
 \node(00) [right  of=d2] {\tiny$\overline{k-1}$};
 \node(00r) [right  of=00] {};
 \node(02) [right  of=00r] {\tiny$2k-1$};
 \node(01) [below  of=02] {\tiny$k-1$};
 \node(03) [above  of=02] {\tiny$3k-1$};
 \node(02r) [right  of=02] {};
 \node(10) [right  of=02r] {\tiny$\overline{0}$};
 \node(10r) [right  of=10] {};
 \node(12) [right  of=10r] {\tiny$k$};
 \node(11) [below  of=12] {\tiny$0$};
 \node(13) [above  of=12] {\tiny$2k$};
 \node(12r) [right  of=12] {};
 \node(20) [right  of=12r] {\tiny$\overline{1}$};
 \node(20r) [right  of=20] {};
 \node(22) [right  of=20r] {\tiny$k+1$};
 \node(21) [below  of=22] {\tiny$1$};
 \node(23) [above  of=22] {\tiny$2k+1$};
 \node(f) [right  of=22] {\tiny$\cdots$};
\draw [->,>=stealth'] (d1)  to  (00) ; 
\draw [->,>=stealth'] (d2)  to  (00) ; 
\draw [->,>=stealth'] (d3)  to  (00) ; 

\draw [->,>=stealth']  (00) to node[above]{\tiny$\alpha_{k-1}$ } (01); 
\draw [->,>=stealth']  (00)  to node[above]{\tiny$\alpha_{2k-1}$} (02); 
\draw [->,>=stealth']  (00)  to node[above]{\tiny$\alpha_{3k-1}$} (03); 
\draw [->,>=stealth']  (01) to node[above]{\tiny$\beta_{k-1}$ } (10); 
\draw [->,>=stealth']  (02) to node[above]{\tiny$\beta_{2k-1}$ } (10); 
\draw [->,>=stealth']  (03) to node[above]{\tiny$\beta_{3k-1}$ } (10);

\draw [->,>=stealth']  (10) to  node[above]{\tiny$\alpha_{0}$ } (11); 
\draw [->,>=stealth']  (10)  to  node[above]{\tiny$\alpha_{k}$ } (12) ; 
\draw [->,>=stealth']  (10)  to  node[above]{\tiny$\alpha_{2k}$ } (13); 
\draw [->,>=stealth']  (11) to node[above]{\tiny$\beta_{0}$ } (20); 
\draw [->,>=stealth']  (12) to node[above]{\tiny$\beta_{k}$ } (20); 
\draw [->,>=stealth']  (13) to node[above]{\tiny$\beta_{2k}$ } (20);

\draw [->,>=stealth']  (20) to  node[above]{\tiny$\alpha_{0}$ } (21); 
\draw [->,>=stealth']  (20)  to  node[above]{\tiny$\alpha_{k+1}$ } (22) ; 
\draw [->,>=stealth']  (20)  to  node[above]{\tiny$\alpha_{2k+1}$ } (23); 
\draw [->,>=stealth']  (21) to (f); 
\draw [->,>=stealth']  (22) to (f); 
\draw [->,>=stealth']  (23) to (f);
\end{tikzpicture}
\end{center}
Let $I_k$ be the ideal of $\kk Q_k$ generated by the elements $\alpha_{i+k+1}\beta_i$,  $\alpha_{i+2k+1}\beta_i$, $\beta_i\alpha_i-\beta_{i+k}\alpha_{i+k}$ and $\beta_i\alpha_i-\beta_{i+2k}\alpha_{i+2k}$ for all $i\in\mathbb{Z}/3k\mathbb{Z}$. We set $\Lambda_k=\kk Q_k/I_k$. This algebra is a selfinjective algebra of finite representation type with tree type $D_4$, frequency $k$ and torsion order $3$. In fact, $\Lambda_k$ is a unique modulo isomorphism basic algebra with such tree type, frequency and torsion order. We denote by $e_x$ the idempotent of $\Lambda_k$ corresponding to the vertex $x\in\mathbb{Z}/k\mathbb{Z}\cup \mathbb{Z}/3k\mathbb{Z}$. We also set $P_x=e_x\Lambda_k$ and $P_{x,y}=\Lambda_ke_x\otimes e_y\Lambda_k$. Thus, $P_x$ is the projective $\Lambda_k$-module corresponding to the idempotent $e_x$ and $P_{x,y}$ is the projective $\Lambda_k$-bimodule corresponding to the idempotent $e_x\otimes e_y$.

The algebra $\Lambda_k$ has Nakayama automorphism $\nu$ of order $3k$ defined by the equalities $\nu(e_i)=e_{i-1}$, $\nu(e_{\overline{i}})=e_{\overline{i-1}}$, $\nu(\alpha_i)=\alpha_{i-1}$ and $\nu(\beta_i)=\beta_{i-1}$. Note that due to \cite{Ric2} the functor $-\otimes_{\Lambda_k}(\Lambda_k)_{\nu}$ commutes with any standard derived equivalence. This means, in particular, that $X_{\nu}\cong X$ in $\Cb_{\Lambda_k}$ for any radical tilting complex $X$.

It is not difficult to see that $P_{\overline{i}}$ ($i\in\mathbb{Z}/k\mathbb{Z}$) form a $0$-spherical sequence $E$ of length $k$ while $P_{i}$ ($i\in\mathbb{Z}/3k\mathbb{Z}$) form a $0$-spherical sequence $E'$ of length $3k$. Moreover, we have $\Hom_{\Kbp_{\Lambda_k}}(E,E')=3k$ and $\Hom_{\Kbp_{\Lambda_k}}(E,E'[l])=0$ for any nonzero integer $l$. Thus, we can apply item \ref{g2} of Theorem \ref{main} to conclude that $T_E$ and $T_{E'}$ generate a subgroup of $\TrPic(\Lambda_k)$ isomorphic to $\sB_{G_2}$. In fact, $T_E$ and $T_{E'}$ are the functors of the tensor multiplication by the complexes of $\Lambda_k$-bimodules 
$$
C_E=\bigoplus\limits_{i\in\mathbb{Z}/k\mathbb{Z}}P_{\overline{i},\overline{i}}\xrightarrow{\mu_E}\Lambda_k\mbox{ and }C_{E'}=\bigoplus\limits_{i\in\mathbb{Z}/3k\mathbb{Z}}P_{i,i}\xrightarrow{\mu_{E'}}\Lambda_k
$$
respectively. In both cases $\Lambda_k$ is placed in the zero degree. The maps $\mu_E$ and $\mu_{E'}$ are defined by the equalities $\mu_E(u\otimes v)=uv$ for $u,v\in P_{\overline{i},\overline{i}}$ and $\mu_{E'}(u\otimes v)=uv$ for $u,v\in P_{i,i}$. We will show that in fact $T_E$ and $T_{E'}$ generate almost the whole group $\TrPic(\Lambda_k)$. But first we will prove some technical lemmas. In fact all details of our proof except Lemma \ref{gener} below will be taken from the proof of the second part of \cite[Theorem 1]{VZ2}. We give them here for the convenience of the reader, but since the difference with the mentioned proof is minor, we will not be very detailed.

\begin{lemma}\label{Pic} One has $\Out(\Lambda_k)\cong \mathbb{Z}/3k\mathbb{Z}\times \kk^*$, where the generator of $\mathbb{Z}/3k\mathbb{Z}$ is the Nakayama automorphism $\nu$ and the elements $\varepsilon\in\kk^*$ corresponds to the class of the automorphism that is identical on $e_i$, $e_{\overline{i}}$, $\alpha_i$ and $\beta_j$ for all integer $i$ and all integer $j$ not divisible by $k$ and sends $\beta_{j}$ to $\varepsilon\beta_j$ if $k\mid j$. In particular, $\Pic_0(\Lambda_k)\cong\kk^*$.
\end{lemma}
\begin{proof} It follows from \cite{Pol} and \cite{GAS} that $\Out(\Lambda_k)\cong \Aut_S(\Lambda_k)/\big(\Inn(\Lambda_k)\cap \Aut_S(\Lambda_k)\big)$, where $\Aut_S(\Lambda_k)$ denotes the set of automorphisms of $\Lambda_k$ that stabilize the subalgebra generated by $e_x$ ($x\in\mathbb{Z}/k\mathbb{Z}\cup \mathbb{Z}/3k\mathbb{Z}$). It is clear that the idempotent $e_0$ can be sent only to an idempotent $e_i$ for some $i\in\mathbb{Z}/3k\mathbb{Z}$ and that its image determines the images of all other idempotents. Thus, modulo the Nakayama automorphism $\nu$ that generate a central subgroup $\mathbb{Z}/3k\mathbb{Z}$ in $\Out(\Lambda_k)$, any element of $\Out(\Lambda_k)$ can be represented by an automorphism fixing $e_x$ for any $x\in\mathbb{Z}/k\mathbb{Z}\cup \mathbb{Z}/3k\mathbb{Z}$. Such an automorphism simply sends $\alpha_i$ to $\kappa_i\alpha_i$ and $\beta_i$ to $\varepsilon_i\beta_i$ for some $\kappa_i,\varepsilon_i\in\kk^*$ ($i\in\mathbb{Z}/3k\mathbb{Z}$) such that $\kappa_i\varepsilon_i=\kappa_{i+k}\varepsilon_{i+k}$ for any $i\in\mathbb{Z}/3k\mathbb{Z}$. It is also not difficult to see that modulo a central element any invertible element $x$ such that the conjugation by $x$ belongs to $\Aut_S(\Lambda_k)$ can be represented in the form $x=\sum\limits_{x\in \mathbb{Z}/k\mathbb{Z}\cup \mathbb{Z}/3k\mathbb{Z}}\lambda_xe_x$ for some $\lambda_x\in\kk^*$. Applying a conjugation by such $x$, we can change the parameters $\kappa_i$ and $\varepsilon_i$ in such a way that $\kappa_i=1$ for any $i\in\mathbb{Z}/3k\mathbb{Z}$ and $\varepsilon_i=1$ for any $i\in\mathbb{Z}/3k\mathbb{Z}$ such that $k\nmid i$. Thus, we get a surjective homomorphism from $\mathbb{Z}/3k\mathbb{Z}\times \kk^*$ to $\Out(\Lambda_k)$ described in the assertion of the lemma. Moreover, one can show that conjugation by $x$ described above cannot turn the image of $\alpha\in\kk^*$ to identical automorphism, because such a conjugation preserves the value of $\prod_{i=0}^{k-1}\frac{\varepsilon_i}{\kappa_i}$. Thus, we get the required isomorphisms.
 \end{proof}

We will denote by $\widehat{\varepsilon}\in\Pic_0$ the image of $\varepsilon\in\kk^*$ under the isomorphism from Lemma \ref{Pic}.

\begin{lemma}\label{relcent} $(T_ET_{E'})^3\cong -\otimes (\Lambda_k)_{(\widehat{-1})^k\nu^{-3}}[5]$.
\end{lemma}
\begin{proof} It follows from the formula for $(T_ET_{E'})^3$ from the previous section that $(T_ET_{E'})^3\big(P\otimes (\Lambda_k)_{\nu^{3}}[-5]\big)\cong P$ for any projective $\Lambda_k$-module $P$, and hence $(T_ET_{E'})^3\circ \big(-\otimes (\Lambda_k)_{\nu^{3}}[5]\big)\in\Pic_0$.
There is a canonical map from $\TrPic(\Lambda_k)$ to the group of isomorphism classes of $\Lambda_k$-bimodules inducing stable autoequivalences
of $\Lambda_k$ (see \cite[Section 3.4]{RZ} for details). This map is injective on $\Pic_0(\Lambda_k)$. At the same
time it sends $T_E$ and $T_{E'}$ to $\Lambda_k$ and $[-5]$ to $\Omega^5_{\Lambda_k^{\rm op}\otimes_\kk\Lambda_k}(\Lambda_k)$. So it is enough to prove that $\Omega^5_{\Lambda_k^{\rm op}\otimes_\kk\Lambda_k}(\Lambda_k)\cong (\Lambda_k)_{(\widehat{-1})^k\nu^{-3}}$, but it follows from \cite[Section 3]{GIV}.
\end{proof}

\begin{lemma}\label{gener}
Let $X\in\Kbp_{\Lambda_k}$ be a tilting complex. Then there is some autoequivalnce $\Phi$ belonging the group generated by $T_E$ and $T_E'$ such that the radical representative of $\Phi X$ is concentrated in one degree.
\end{lemma}
\begin{proof} It is enough to prove that for any tilting complex $X$ with $L(X)>1$ there is some $\Phi$ in  the group generated by $T_E$ and $T_E'$ such that $L(\Phi X)<L(X)$. We may assume that $X$ is radical and is concentrated in degrees from $0$ to $L-1$, where $L=L(X)>1$. Then either $P_i$ or $P_{\overline{i}}$ is a direct summand of $X_{L-1}$ for some $i\in\mathbb{Z}/3k\mathbb{Z}$. We will consider the case where $P_i$ is a direct summand of $X_{L-1}$, the second case is analogous. It follows from the condition $X_{\nu}\cong X$ that $P_i$ is a direct summand of $X_{L-1}$ for any $i\in\mathbb{Z}/3k\mathbb{Z}$.
If $P_i$ is a direct summand of $X_0$ for some $i\in\mathbb{Z}/3k\mathbb{Z}$, then the map $\alpha_{i}\beta_{i-1}:P_{i-1}\rightarrow P_i$ induces a map from $X_{L-1}$ to $X_0$ that is annihilated by $d_1$ and $d_{L-1}$, because $X$ is radical and $J_{\Lambda_k}\alpha_{i}\beta_{i-1}=\alpha_{i}\beta_{i-1}J_{\Lambda_k}=0$. Thus, we obtain a nonzero morphism from $X$ to $X[1-L]$ in $\Kbp_{\Lambda_k}$ that is impossible. Then $P_{\overline{i}}$ is a direct summand of $X_{0}$ for some $i\in\mathbb{Z}/k\mathbb{Z}$ and one can shows that in fact $P_{\overline{i}}$ is a direct summand of $X_{0}$ and is not a direct summand of $X_{L-1}$ for any $i\in\mathbb{Z}/k\mathbb{Z}$. Let us apply $T_{E'}$ to $X$. Without loss of generality, we may assume that $T_{E'}X$ is radical. It follows from $X\cong -\otimes_{\Lambda_k}C_{E'}$ that $T_{E'}X$ is concentrated in degrees from $-1$ to $L-1$, all direct summands of $(T_{E'}X)_{-1}$ are isomorphic to $P_i$ for some $i\in\mathbb{Z}/3k\mathbb{Z}$ and all direct summands of $(T_{E'}X)_{L-1}$ are isomorphic to direct summands of $X_{L-1}$, i.e. to $P_i$ for some $i\in\mathbb{Z}/3k\mathbb{Z}$. On the other hand, $\Hom_{\Kbp_{\Lambda_k}}(T_{E'}X[L-1],P_i)=\Hom_{\Kbp_{\Lambda_k}}(X[L-1],P_i[-1])=0$, and hence $P_i$ cannot be a direct summand of $(T_{E'}X)_{L-1}$. Thus, $T_{E'}X$ is concentrated in degrees from $-1$ to $L-2$. If $(T_{E'}X)_{-1}=0$, then the required assertion is proved. In the opposite case $P_{i}$ cannot be a direct summand of $X_{l-2}$ for $i\in\mathbb{Z}/3k\mathbb{Z}$, and hence all direct summands of $X$ are isomorphic to $P_{\overline{i}}$ for $i\in\mathbb{Z}/k\mathbb{Z}$. Let us apply $T_E$ to $T_{E'}X$ assuming that the resulting complex is radical. The same argument as above shows that $T_ET_{E'}X$ is concentrated in degrees from $-2$ to $L-3$ and if $(T_ET_{E'}X)_{-2}\not=0$, then all direct summands of $(T_ET_{E'}X)_{L-3}$ are isomorphic to $P_i$ with $i\in\mathbb{Z}/3k\mathbb{Z}$. Continuing this process, we get that if $L(\Phi X)\ge L(X)$ for any $\Phi$ from the group generated by $T_E$ and $T_E'$, then we may assume that $(T_ET_{E'})^3X$ is concentrated in degrees from $-6$ to $L-7$.
On the other hand, $(T_ET_{E'})^3\cong -\otimes (\Lambda_k)_{(\widehat{-1})^k\nu^{-3}}[5]$ by Lemma \ref{relcent}, and hence $(T_ET_{E'})^3X\cong X_{(\widehat{-1})^k}[5]$ is concentrated in degrees from $-5$ to $L-6$. The obtained contradiction finishes the proof of the lemma.
 \end{proof}

\begin{theorem} $\TrPic(\Lambda_k)\cong (\sB_G\times\mathbb{Z}\times\mathbb{Z}/3k\mathbb{Z}\times\kk^*)/\big(\Delta_G^{-1},5,3,(\widehat{-1})^k\big)$. Under this isomorphism the standard generators $\sigma_1$ and $\sigma_2$ of $\sB_G$ correspond to $T_E$ and $T_{E'}$, the generator $1\in\mathbb{Z}$ corresponds to the shift functor $[1]$, the generator $1\in\mathbb{Z}/3k\mathbb{Z}$ corresponds to the Nakayama automorphism $\nu$ and $\varepsilon\in\kk^*$ corresponds to the automorphism $\widehat{\varepsilon}\in\Pic_0(\Lambda_k)$.
\end{theorem}
\begin{proof} One can directly construct an isomorphism between $C\otimes_{\Lambda_k}(\Lambda_k)_{\varepsilon^{-1}}$ and $(\Lambda_k)_{\varepsilon^{-1}}\otimes_{\Lambda_k}C$ for $C=C_E,C_{E'}$ (see the proof of \cite[Proposition 3]{VZ2}). Since the shift functor and the Nakayama functor commute with any standard derived equivalence we get a homomorphism $\phi:\sB_G\times\mathbb{Z}\times\mathbb{Z}/3k\mathbb{Z}\times\kk^*\rightarrow \TrPic(\Lambda_k)$ described in the theorem. Let us prove that the kernel of this homomorphism is generated by the element $(\Delta_G^{-1},5,3,(\widehat{-1})^k)$ that belongs to the kernel by Lemma \ref{relcent}. Suppose that the element $(w,a,b,\widehat{\varepsilon})$ belongs to $\Ker\phi$. Then $\phi(w)$ commutes with $\phi(w')$ for any $w'\in \sB_G$. Since $\phi|_{\sB_G}$ is injective by item \ref{g2} of Theorem \ref{main}, we have $w\in Z_G$, i.e. $w=\Delta_G^t$ for some integer $t$. Then
$$(w,a,b,\widehat{\varepsilon})\left(\Delta_G^{-t},5t,3t,(\widehat{-1})^{kt}\right)=\left(1,a+5t,b+3t,(\widehat{-1})^{kt}\widehat{\varepsilon}\right)\in\Ker\phi.$$
But since the intersection of $\Pic(\Lambda_k)$ with the subgroup of $\TrPic(\Lambda_k)$ generated by the shift is trivial, we have $a=-5t$, $3k\mid b-3t$ and $(-1)^{kt}\varepsilon=1$ by Lemma \ref{Pic}, i.e. $(w,a,b,\widehat{\varepsilon})=\left(\Delta_G^{-1},5,3,(\widehat{-1})^{k}\right)^{-t}$.

It remains to prove that $\phi$ is surjective. But it follows from Lemma \ref{gener} that for any $\Phi\in\TrPic(\Lambda_k)$ there is some $w\in\sB_G$ such that $\phi(w)\Phi$ belongs to the direct product of $\Pic(\Lambda_k)$ and the subgroup of $\TrPic(\Lambda_k)$ generated by the shift functor. Since the $\Im\phi$ contains $\Pic(\Lambda_k)$ and the shift functor, we have $\Phi\in\Im\phi$ and the theorem is proved.
 \end{proof}

{\bf Acknowledgements.} The work was supported by  RFBR according to the research project 18-31-20004, by the President’s ”Program
Support of Young Russian Scientists”  according to the research project MK-2262.2019.1 and in part by Young Russian Mathematics award.

\end{document}